\documentclass[11pt,a4paper,reqno]{amsart}

\usepackage[utf8]{inputenc}
\usepackage[T1]{fontenc}
\usepackage[margin=3cm]{geometry}
\usepackage{amsmath}
\usepackage{amsthm}
\usepackage{amssymb}
\usepackage{amsopn}
\usepackage{enumitem}
\usepackage{xcolor}
\usepackage{mathtools}
\usepackage{doi}
\usepackage{xfrac}
\usepackage{graphicx}
\usepackage{bbm}
\usepackage{relsize}
\usepackage[labelformat=simple]{subcaption}

\usepackage[normalem]{ulem}

\makeatletter
\def\paragraph{\@startsection{paragraph}{4}%
  \z@\z@{-\fontdimen2\font}%
  {\normalfont\bfseries}}
\makeatother

\theoremstyle{plain}
\newtheorem{theorem}{Theorem}[section]
\newtheorem{prop}[theorem]{Proposition}

\newtheorem{lemma}[theorem]{Lemma}

\theoremstyle{definition}
\newtheorem{definition}[theorem]{Definition}

\newtheorem{assumption}{Assumption}

\theoremstyle{remark}
\newtheorem{remark}[theorem]{Remark}

\numberwithin{equation}{section}

\DeclareMathOperator{\co}{co}

\DeclareMathOperator{\supp}{supp}

\DeclareFontFamily{OT1}{pzc}{}
\DeclareFontShape{OT1}{pzc}{m}{it}{<-> s * [1.40] pzcmi7t}{}
\DeclareMathAlphabet{\mathpzc}{OT1}{pzc}{m}{it}

\newcommand{\Dm}[1]{D_{1#1}}
\newcommand{\Ds}[1]{D_{2#1}}

\newcommand{\R}{\mathbb{R}}
\newcommand{\Rie}{\ensuremath{J}}

\newcommand{\Rc}{\ensuremath{\mathcal{R}}}

\newcommand{\Hc}{\mathcal{H}}
\newcommand{\Hcc}{\ensuremath{\dot{\mathcal{H}}^1}}
\newcommand{\La}{\Lambda}

\newcommand{\Z}{\mathbb{Z}}
\newcommand{\C}{\mathbb{C}}

\usepackage{color}
\newcommand{\cb}{}

\definecolor{fgreen}{HTML}{2ECC40}

\renewcommand{\co}[1]{{ #1}}

\newcommand{\RFgL} 
{
       \Sigma^{\nabla}(h)
}
\newcommand{\RFgR}[1] 
{
       \frac{1}{2} \int_0^1 \dot{h}(t)^2 \, dt - #1|\{t \in [0,1] : h(t) = 0 \}|
}

\newcommand{\RFlL} 
{
       \Sigma^{\Delta}(h)
}

\newcommand{\RFlR}[1] 
{
       \frac{1}{2} \int_0^1 \ddot{h}(t)^2 \, dt - #1|\{t \in [0,1] : h(t) = 0 \}|
}

\newcommand{\<}[2]{\ensuremath{\langle #1,\,#2\rangle}}

\newcommand{\E}{\mathcal{E}}

\newcommand{\B}{\mathcal{B}}

\newcommand{\G}{\mathcal{G}}

\makeatletter
\def\subsubsection{\@startsection{subsubsection}{3}%
  \z@{.5\linespacing\@plus.7\linespacing}{-.5em}%
  {\normalfont\bfseries}}

\begin{document}

\title{Analysis of Cell Size Effects in Atomistic Crack Propagation}


\author{Maciej Buze}
\author{Thomas Hudson}
\author{Christoph Ortner}
\address{Mathematics Institute\\
  Zeeman Building\\
  University of Warwick\\
  Coventry\\
  CV4 7AL\\
  United Kingdom}
\curraddr{}
\email[M.~Buze]{m.buze@warwick.ac.uk}
\email[T.~Hudson]{t.hudson.1@warwick.ac.uk}
\email[C.~Ortner]{c.ortner@warwick.ac.uk}

\thanks{MB is supported by EPSRC as part of the MASDOC DTC, Grant No. EP/HO23364/1.}
\thanks{TH is supported by the Leverhulme Trust through Early Career Fellowship ECF-2016-526}
\thanks{CO is  supported by EPSRC Grant EP/R043612/1}

\subjclass[2010]{65L20, 70C20, 74A45, 74G20, 74G40, 74G60, 74G65}

\keywords{crystal lattices, defects, crack propagation, regularity, bifurcation theory, convergence rates.}

\date{\today}

\dedicatory{}

\begin{abstract}
  We consider crack propagation in a crystalline material in terms of bifurcation
  analysis. We provide evidence that the stress intensity factor is a natural
  bifurcation parameter, and that the resulting bifurcation diagram is a
  periodic ``snaking curve''. We then prove qualitative properties of the
  equilibria and convergence rates of finite-cell approximations to the
  ``exact'' bifurcation diagram.
\end{abstract}
\maketitle

\section{Introduction}
A fundamental task of materials modelling is to understand the process of failure, which is often facilitated by crack propagation. Cracks (and other defects) initiate and propagate via atomistic mechanisms, which renders the task of creating accurate and efficient simulations of this phenomenon on a large scale particularly difficult \cite{Bitzek2015}. In addition, many of the simulation techniques in operation today rely on simplifying assumptions which are often phenomenological; for example, it is generally unclear under which conditions continuum models become invalid \cite{GSHY}.

There is thus a need for a robust mathematical theory of crack propagation at an atomistic scale, providing a rigorous grounding for a subsequent study of bottom-up multiscale and coarse-grained models. In \cite{2018-antiplanecrack} we began to lay the foundation of such a theory by formulating the equilibration problem on a lattice in the presence of a crack as a variational problem on an appropriate discrete Sobolev space, and establishing existence, local uniqueness and stability of equilibrium displacements for small loading parameters. Crucially, we also established decay properties of the lattice Green's function in crack geometry, which enabled us to prove qualitatively sharp far-field decay estimates of the atomistic core contribution to the equilibrium fields in order to quantify the ``range'' of atomistic effects. This work relied on and extended the recent rigorously formalised atomistic theory of single localised defects in crystalline structures \cite{EOS2016,HO2014,2017-bcscrew}.

The purpose of the present work is to go beyond the small-loading regime and introduce a key component missing in \cite{2018-antiplanecrack}, demonstrating that crack propagation, facilitated by bond-breaking events, can be described in the framework of \cite{2018-antiplanecrack}. The mathematical tools we exploit to do so are taken from bifurcation theory in Banach spaces \cite{cliffe_spence_tavener_2000}. While this idea has already been explored numerically in \cite{Li2013,Li2014}, a key new conceptual insight is that the \textit{stress intensity factor} (SIF), which acts as a measure of stability in continuum fracture can be interpreted as the ``loading parameter'' on the atomistic crack through the far-field boundary condition allowing us to obtain rigorous results about cell size effects.

More specifically, 
we model the equilibration of an atomistic crack embedded in an infinite homogeneous crystal as a variational problem with the continuum linearised elasticity (CLE) solution as the far-field boundary condition. The SIF enters the model as a scaling parameter multiplying the CLE solution, and so varying this naturally leads to a bifurcation diagram. Moreover, the fact that the CLE crack equilibrium displacement does not belong to the energy space suggests that the bifurcation diagram consists solely of regular points and quadratic fold points, at which the equilibria found transition from being linearly stable to linearly unstable (or vice versa).

This observation and the numerical evidence we obtain together motivate structural assumptions on the bifurcation diagram: we assume (and confirm numerically) that it is a `snaking curve' \cite{TAYLOR201014} with the stability of solutions changing at each bifurcation point. In particular, under our assumptions, a jump from one stable segment to another captures the propagation of the crack through one lattice cell, with the unstable segment that is crossed in that jump being a corresponding saddle point, which represents the energetic barrier which must be overcome for crack propagation to occur at a given value of the SIF. This allows us to capture the phenomenon of `lattice trapping' \cite{Thomson_1971,gumbsch_cannon_2000}, a term which refers to the idea that in discrete models of fracture there can exist a range of values of SIF for  which the crack remains locally stable despite being above or below the critical Griffith stress.

As in \cite{2018-antiplanecrack}, we avoid significant technicalities by restricting the analysis to a two--dimensional square lattice with nearest neighbour pair interaction. The notable difference in the models considered is that in \cite{2018-antiplanecrack}, in order to prove that the variational problem is well-posed, the bonds crossing the crack were explicitly removed from the interaction; by contrast, in the present paper they are included in the interaction range, and instead, the fact that they are effectively broken is encoded in the interatomic potential. This gives rise to a physically realistic periodic bifurcation diagram, for which we subsequently prove regularity results both in terms of its smoothness as a submanifold of an appropriate space, as well as uniform spatial regularity of the equilibria along the corresponding solution path.

Our results for the infinite lattice model naturally lead to an investigation of the numerical approximation of these solutions on a finite-domain, and we use the technical tools established in \cite{2018-antiplanecrack} to establish sharp convergence rates as the domain radius tends to infinity. A notable novelty is that our results apply uniformly to finite segments of the bifurcation diagram; moreover, we establish a superconvergence result for the critical values of the SIF at which fold points occur. Since the unstable segments of {\cb the} bifurcation diagram correspond to index--1 saddle points of the energy, our work in this regard also extends the convergence results of \cite{2018-uniform} for saddle point configurations of point defects and suggests possible future extensions to a full transition state analysis \cite{Eyr,HTB90,Wigner1997,Berglund13,2018-entropy}. \\

\paragraph{Outline:} In Section \ref{main} we provide a detailed motivation for our work, introducing the model for crack propagation in the anti-plane setup, describing the underlying assumptions, and providing a statement of the main results about the model and its numerical approximation. In Section~\ref{mot} we give a brief overview of the continuum mechanics context and describe how it motivates our work, and in Section~\ref{setup} we discuss the discrete kinematics of the atomistic model. Then, in Section~\ref{model} the key assumptions are presented and discussed, and the novel components of the theory, in particular the role of the stress intensity factor, are highlighted. The main results of the paper are also stated. Section \ref{sec:approx} is dedicated to the finite-domain approximation of the problem, with sharp convergence results stated, including the superconvergence result for the critical values of the stress intensity factor. In Section \ref{numerics} we present a numerical setup employed to compute bifurcation paths, enabling us to numerically verify the sharpness of our results with respect to regularity and rate of convergence. Section \ref{conclusion} then provides a discussion about the significance of our results, and the proofs of the main results are given in Section \ref{proofs}.

\clearpage

\section{Main results}\label{main}
\subsection{Motivation}\label{mot}

The principal motivation for our work stems from the following limitation of the continuum elasticity approaches to static crack problems. Consider a domain $\mathcal{D} \subset \R^2$, representing a cross-section of a three-dimensional elastic body, with a crack set $\Gamma_{\mathcal{D}}\subset\mathcal{D}$. Given a material-specific \textit{strain energy density function} $W\,:\,\R^{2\times k} \to \R\cup\{+\infty\}$ (where $k \in \{1,2,3\}$ depending on the loading mode), c.f. \cite{landau1989theory}, one can hope to find a non-trivial equilibrium displacement $u\,:\,\mathcal{D} \to \R^k$ accommodating the presence of a crack by minimising the continuum energy given by
\[
E(u):= \int_{\mathcal{D}\setminus \Gamma_{\mathcal{D}}} W(\nabla u)\,dx,
\]
over a suitable function space. In line with CLE, one can approximate $W$ by its expansion around zero to second order and obtain the associated equilibrium equation
\begin{align}
-{\rm div}\left(\C\,\colon\nabla u\right) &= 0 \quad\text{in } \mathcal{D}\setminus\Gamma_{\mathcal{D}},\label{cle1}\\
\left(\C\,\colon\nabla u\right)\nu &= 0 \quad\text{on }\Gamma_{\mathcal{D}}\label{cle3},
\end{align}
supplied with a suitable boundary condition coupling to the bulk \cite{freund_1990}. Here $\C$ is the \textit{elasticity tensor} with entries $\mathbb{C}^{j\beta}_{i\alpha} := \partial_{i\alpha j\beta} W(0)$.

It is well-known that regardless of the details of the geometry of $\mathcal{D}$ and $\Gamma_{\mathcal{D}}$, near the crack tip, the gradients of solutions to \eqref{cle1}-\eqref{cle3} exhibit a persistent $1/\sqrt{r}$ behaviour, where $r$ is the distance from the crack tip, c.f. \cite{rice1968mathematical}. The singularity at the crack tip implies the failure of CLE to accurately describe a small region around it where atomistic (nonlinear and discrete) effects dominate. This near-tip nonlinear zone is argued to exhibit \textit{autonomy} \cite{freund_1990,Broberg99}, meaning that the state of the system in the vicinity of the singular field is determined uniquely by the value of the SIF, and therefore systems with the same SIF but different geometries will behave similarly within the near-tip nonlinear zone.

In order to better understand the microscopic features of this zone, we may exploit the spatial invariance of elasticity and zoom in on the region near the crack tip by performing a spatial rescaling $R\,u(x/R)$, which leads to a simplified geometry of an infinite domain with a half-infinite straight crack line, as illustrated in Figure \ref{fig:crack_blowup}.
\begin{figure}[!htbp]
  \begin{subfigure}[t]{.7\textwidth}
    \centering
    \includegraphics[width=\linewidth]{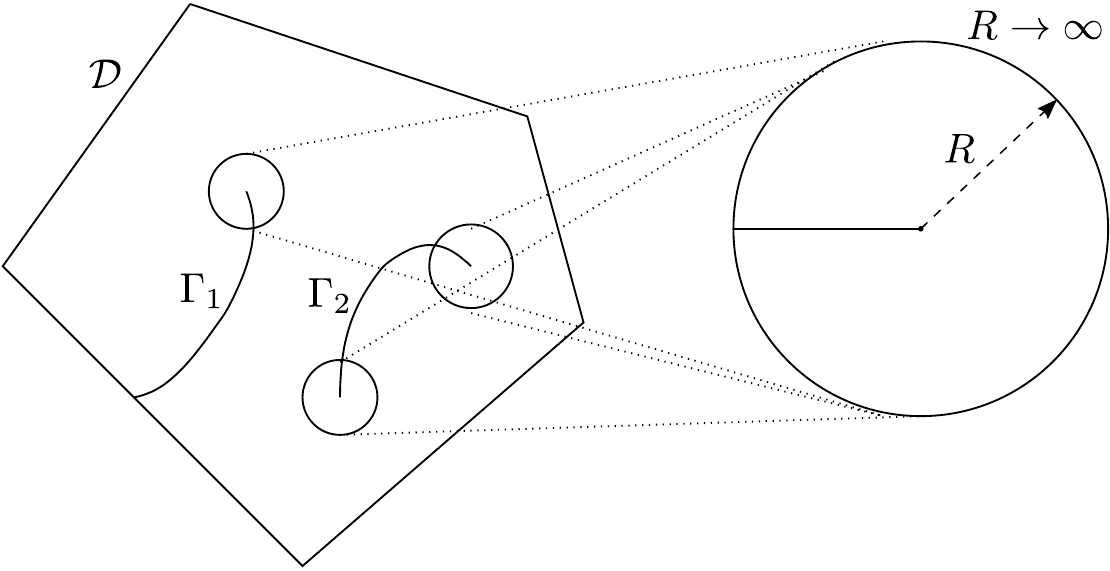}
    \end{subfigure}
  \caption{A schematic illustration of the setup. A domain $\mathcal{D}$ with the crack set $\Gamma_{\mathcal{D}} = \Gamma_1\cup\Gamma_2$. The autonomy of the crack implies we can zoom in on each crack tip to obtain a simplified geometry of a ball of radius $R$. In the limit $R\to \infty$ of the spatial rescaling we obtain a domain $\R^2$ and a half-infinite crack $\Gamma_0$.}
\label{fig:crack_blowup}
\end{figure}

In what follows we focus on Mode III cracks, restricting to anti-plane displacements $u:\R^2 \to \R$. Under natural assumptions on the stored energy density $W$ which result from coupling it with a corresponding interatomic potential, as discussed in \cite{2018-antiplanecrack}, the set of equations \eqref{cle1}-\eqref{cle3} reduces in the simplified geometry to
\begin{align}
-\Delta u &= 0 \quad\text{in } \R^2\setminus\Gamma_0,\label{cle4}\\
\nabla u\cdot\nu &= 0 \quad\text{on }\Gamma_0\label{cle5},
\end{align}
where
\begin{equation}\label{gamma0}
 \Gamma_0 := \{(x_1,0)\,|\,x_1 \leq 0\}.
\end{equation}
This PDE has a canonical solution, as discussed in e.g. \cite{SJ12}, given by
\begin{equation}\label{upred}
\hat{u}_{k}(x) = k\,\sqrt{r}\sin\textstyle{\frac{\theta}{2}},
\end{equation}
with $(r,\theta)$ representing standard cylindrical polar coordinates centred at the crack tip. The scalar parameter $k$ corresponds to the (rescaled) stress intensity factor (SIF) \cite{lawn_1993}.

As we increase the spatial rescaling parameter $R$, we eventually approach the atomic lengthscale at which the hypothesis that the material behaves as a continuum breaks down. We must thus speak of atomic displacements and finite differences rather than differential operators, and consider an atomistic model supplied with the function $\hat{u}_k$ as a far-field boundary condition. As such, in the model we describe below, the function $\hat{u}_k$ will act as a CLE `predictor', representing the behaviour of the material in the far field away from the crack tip.

\subsection{Discrete kinematics}\label{setup}
The atomistic setup is similar to the one introduced in \cite{2018-antiplanecrack}; here, we recall it in detail and highlight new concepts. Let $\La$ denote the shifted two dimensional square lattice defined as
\[ \La := \big\{l - (\tfrac{1}{2},\tfrac{1}{2})\,\big|\,l \in \Z^2\big\}.\]

We consider a crack opening along $\Gamma_0$ defined in \eqref{gamma0}, and distinguish the lines that include lattice points directly above and below $\Gamma_0$. These are defined as
\[
\Gamma_{\pm} := \big\{m \in \La\,\big|\, m_1 < 0\,\text{ and }\,m_2 = \pm \textstyle{\frac{1}{2}} \big\}
\]
and we refer to Figure \ref{fig:lattice_gamma} for a visualisation of the setup.
\begin{figure}[!htbp]
  \begin{subfigure}[t]{.45\textwidth}
    \centering
    \includegraphics[width=\linewidth]{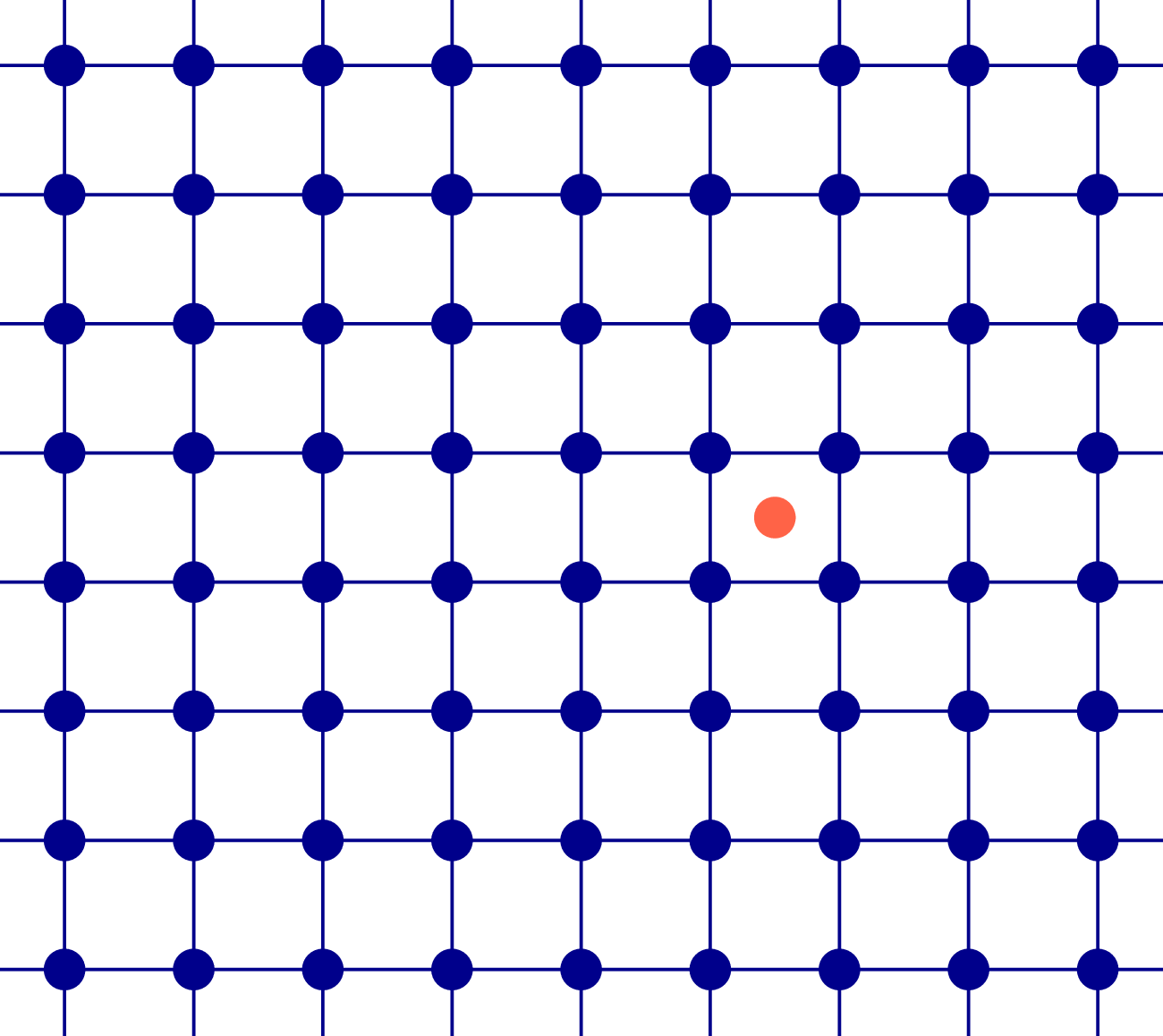}
    \caption{$\,$}\label{fig::fullbonds}
  \end{subfigure}
	\quad
  \begin{subfigure}[t]{.45\textwidth}
    \centering
    \includegraphics[width=\linewidth]{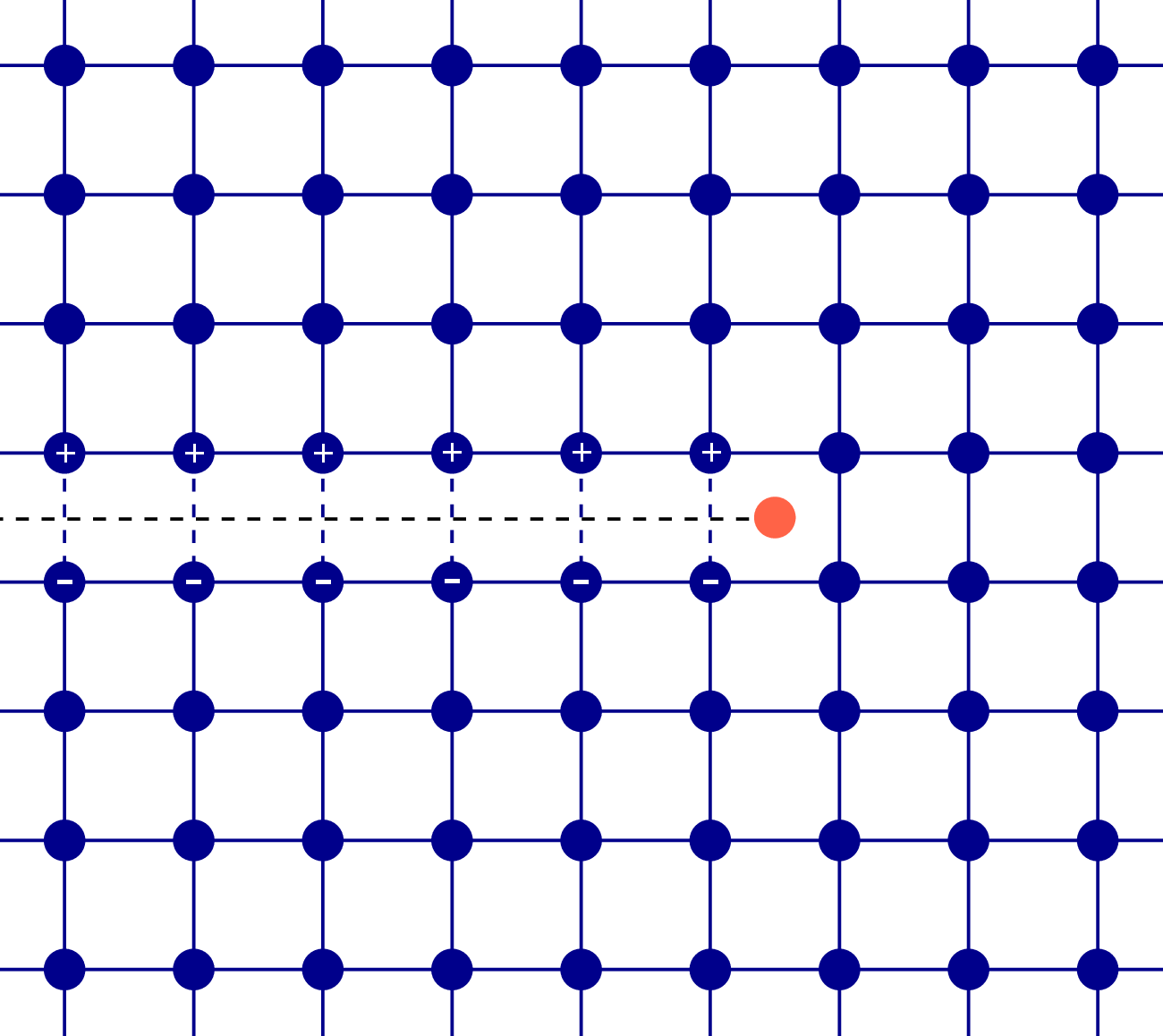}
    \caption{$\,$}\label{fig::cutbonds}
  \end{subfigure}
  \caption{The geometry of the problem with and without the bonds across the crack. The (predicted) crack tip depicted by a red dot. In \subref{fig::cutbonds} the crack cut $\Gamma_0$ from \eqref{gamma0} is shown as a dashed black line and the lattice points on $\Gamma_+$ and $\Gamma_-$ are highlighted.}
\label{fig:lattice_gamma}
\end{figure}

For the purposes of our analysis, it is helpful to consider two notions of interaction neighbourhood for lattice points. First, the nearest neighbour (NN) directions of the homogeneous square lattice are given by
\[
\Rc = \left\{\pm e_1, \pm e_2\right\}.
\]
Second, we modify these interaction neighbourhoods by disregarding the directions across the crack, as these bonds are effectively already broken; for any $m \in \La$, we therefore define
\begin{equation}\label{Rcm}
\tilde{\Rc}(m) := \begin{cases} \Rc \quad & \text{for }\, m \not\in (\Gamma_+\cup\Gamma_-), \\
						\Rc \setminus \left\{\mp e_2\right\}\quad & \text{for }\, m \in \Gamma_{\pm}.\end{cases}
\end{equation}
For an anti--plane displacement defined on the lattice $u\,\colon \,\La \to \R$, we define the finite difference operator as $D_{\rho}u(x) := u(x+\rho) - u(x)$ and introduce two notions of the discrete gradient, denoted by $Du(m),\, \tilde{D}u(m)\, \in \R^{\Rc}$ and defined as
\begin{equation}\label{dgrad}
\big(Du(m)\big)_{\rho} {\, \cb :=\, } D_{\rho}u(m)\text{ and } \big(\tilde{D}u(m)\big)_{\rho} := \begin{cases}
D_{\rho}u(m)\quad &\text{if }\rho \in \tilde{\Rc}(m),\\
0\quad&\text{if }\rho \not\in\tilde{\Rc}(m).
\end{cases}
\end{equation}
{\cb We note that since $|\Rc| = 4$, $\R^{\Rc}$ is a $4$-dimensional space indexed by each member of $\Rc$. It therefore follows from \eqref{dgrad} that $Du$ corresponds to homogeneous NN interactions, whereas $\tilde{D}u$ reflects a defective lattice, as when $m\in \Gamma_{\pm}$, the components of $\tilde{D}u(m)$ which correspond to erased lattice directions are always zero.}

The removal of NN bonds and subsequent introduction of the discrete gradient operator $\tilde{D}$ allows us to define the appropriate discrete energy space (discrete Sobolev space) for handling arbitrarily large differences in the far--field displacements across the crack,
\begin{equation}\label{Hcc}
\begin{gathered}
  \Hcc := \left\{u\,\colon\,\La \to \R \;|\; \tilde{D}u \in \ell^2\;\text{and}\; \textstyle{u(\frac{1}{2},\frac{1}{2})} = 0 \right\},\\
  \begin{aligned}
\text{which has associated norm} \qquad & \|u\|_{\Hcc} := \|\tilde{D}u\|_{\ell^2} = \left(\sum_{m\in\La}|\tilde{D}u(m)|^2\right)^{\sfrac{1}{2}}  \\
\text{ and inner product} \qquad & (u,v)_{\Hcc} := \sum_{m \in \La} \tilde{D}u(m) \cdot \tilde{D}v(m).
\end{aligned}
\end{gathered}
\end{equation}
The choice to restrict $\textstyle{u(\frac{1}{2},\frac{1}{2})}=  0$ ensures that only one constant displacement lies in the space, making $\|\cdot\|_{\Hcc}$ a norm.

It is also helpful to introduce the space of compactly supported displacements,
\[
\Hc^{\rm c} := \{ u\,:\,\La \to \R\,|\,\supp(Du)\,\text{ is compact}\}.
\]
{\cb \begin{remark}
To avoid future confusion, we note that compared to \cite{2018-antiplanecrack} the definitions of $D$ and $\tilde{D}$ have been swapped to accommodate the differing nature of the two papers. In \cite{2018-antiplanecrack} the interactions across the crack are always explicitly excluded, leaving little need for this explicit distinction. In the present work the distinction is crucial and we opted for $D$ to denote the usual intuitive notion of the discrete gradient. Furthermore, a similar change in notation occurs for $\Rc$ and $\tilde{\Rc}$. Note, however, that in both papers the definition of $\Hcc$ remains the same. 
\end{remark}}

\subsection{Analysis of the model}\label{model}
We modify the theory developed in \cite{2018-antiplanecrack} for a small-load anti-plane crack to frame it in the context of bifurcation theory. We consider the energy difference functional $\E\,:\,\Hcc \times \R \to \R$ supplied with CLE solution as a far--field boundary condition:
\begin{equation}\label{energy}
\E(u,k) = \sum_{m \in \La} V\big(D\hat{u}_{k}(m) + Du(m)\big) - V\big(D\hat{u}_{k}(m)\big).
\end{equation}
Here $V\,\colon\,\R^{\Rc} \to \R$ is a suitable interatomic site potential and $\hat{u}_{k}\,\colon\,\La\to\R$ is the CLE predictor, introduced in \eqref{upred}. The function $u \in \Hcc$ is a core correction, thus the total displacement is given by $\hat{u}_{k} + u$.


We assume the site potential to be a NN pair-potential of the form
\begin{equation}\label{pair-potential}
V(Du(m)) = \sum_{\rho \in \Rc} \phi\big(D_\rho u(m)\big),
\end{equation}
with $\phi \in C^\alpha(\R)$ for $\alpha \geq 5$. Assuming that the lattice is reflection symmetric in the anti--plane direction, we assume without loss of generality that
\begin{equation*}
  \phi(0) = 0,\quad\phi'(0)=0,\quad\phi''(0) = 1,\quad\text{and}\quad\phi'''(0)=0.
\end{equation*}
The first and third assumptions may be made by subtracting an appropriate constant and rescaling the potential, while the second and fourth follow from the assumption of anti-plane symmetry, as discussed in \cite{2018-antiplanecrack}.

Note that we employ the homogeneous discrete gradient operator $Du$ in the definition of $\E$, while the `crack-aware' gradient $\tilde{D}u$ is used to define the space $\Hcc$. This is helpful in the context of capturing crack propagation, since it enables us to consider displacements with arbitrarily large strains across the crack, but raises the issue that for any $m \in \Gamma_{\pm}$ and $\rho \not\in \tilde{\Rc}(m)$ crossing the crack surface, we have $D_{\rho}\hat{u}(m) \sim |m|^{1/2}$. Thus, in order for such $\E$ to be well-defined on $\Hcc\times \R$, we further assume that the pair-potential satisfies
\begin{equation}\label{V-new-ass}
\text{there exists }\,R_{\phi} >0\,\text{ such that }\, \phi'(r) = 0\quad\forall r\,\text{ with } |r| \geq R_{\phi}.
\end{equation}
Such an assumption is sufficient and simplifies the exposition, but is by no means a necessary condition and can be easily replaced by an appropriate decay property (e.g. exponential or sufficiently high algebraic decay). In particular, under this assumption, we {\cb firstly prove} the following result.

\begin{theorem}\label{Eu-well-def}
The energy difference functional $\E$ expressed in \eqref{energy} is well-defined on $\Hcc\times \R$ and is $\alpha$-times continuously differentiable.
\end{theorem}
{\cb The proof is given in Section \ref{p-model} and mostly relies on the analogous result in \cite{2018-antiplanecrack}, with the extra work needed to handle the now-included bonds across the crack.}

The inclusion of the stress intensity factor $k$ as a variable in the definition of $\E$ allows us to employ bifurcation analysis to describe the propagation of the crack as a series of bifurcations,  which we view as corresponding to bond-breaking events.

The primary task of our analysis is to characterise the set of critical points of the energy, $S$, defined as
\begin{equation}\label{Sset}
S:= \big\{(u,k) \in \Hcc \times \R\,\big|\, \delta_u \E(u,k) = 0 \in (\Hcc)^*\big\},
\end{equation}
where $\delta_u \E\,:\,\Hcc \times \R \to (\Hcc)^*$ is the partial Fr\'echet derivative given by
\[
\<{\delta_u\E(u,k)}{v} = \sum_{m\in\La}\delta V (D\hat{u}_{k}(m) + Du(m)) \cdot Dv(m).
\]

For future reference, we summarize our notation for linear and multi-linear forms, in particular defining the meaning of $\<{\delta_u\E(u,k)}{v}$. For any $n$-linear form $L$, we write $L[v_1,\dots,v_n]$ to denote its evaluation at $v_1,\dots, v_n$ and if $m < n$, then $L[v_1,\dots,v_m]$ is the $(n-m)$-linear form $(w_1,\dots,w_{n-m}) \mapsto L[v_1,\dots,v_m,w_1,\dots,w_{n-m}]$. For the sake of readability and only when there is no risk of confusion, we often write $\<{L}{v_1}$ for linear forms and $\<{L_1 v_1}{v_2}$ as well as $L_1 v_1 = L_1[v_1]$ for bilinear forms.

It is of particular interest to compute continuous paths contained in $S$, as it allows to characterise the response of the model to variations in SIF. This is often possible if we are able to identify one particular pair, say $(\bar{u}_0,\bar{k}_0) \in S$ and it can be further shown that it is a \textit{regular point}, by which we mean
\begin{equation}\label{regp}
H_0 := \delta_{uu}^2 \E(\bar{u}_0,\bar{k}_0)\,:\, \Hcc \to (\Hcc)^*\,\text{ is an isomorphism.}
\end{equation}
In this case, a standard application of the Implicit Function Theorem \cite{serge} yields existence of a locally unique path of solutions $(\bar{u}_s,\bar{k}_s)$ in the vicinity of $(\bar{u}_0,\bar{k}_0)$ which we will assume to be parametrised with an index $s\in\R$; exactly this strategy was used in \cite{2018-antiplanecrack} to show existence of solutions in a static crack problem with crack bonds removed from the definition of $\E$, for $k$ small enough. We will set
\begin{equation}\label{def-hess}
H_s := \delta_{uu}^2 \E(\bar{u}_s,\bar{k}_s)\,:\, \Hcc \to (\Hcc)^*.
\end{equation}

As we will see in the numerical examples of Section~\ref{numerics}, beyond some critical value of $k$, bifurcations of the following type begin to occur.
\begin{definition}\label{deffp}
A \textit{(simple quadratic) fold point} occurs at $(\bar{u}_b,\bar{k}_b) \in S$ if there exists $\gamma_b \in \Hcc$ such that ${\rm Ker}(H_b) = {\rm span}\{\gamma_b\}$,
\begin{align}
\delta^2_{u k} \E(\bar{u}_b,\bar{k}_b)[\gamma_b, 1] &\neq 0, \label{nondeg2} \\
\label{nondeg1}
\delta^3_{u u u} \E(\bar{u}_b,\bar{k}_b)[\gamma_b, \gamma_b, \gamma_b] &\neq 0,
\end{align}
with formulae for these variations of energy given in \eqref{Euk} and \eqref{Euuu}, respectively.
\end{definition}

\begin{figure}[!htb]
  \begin{subfigure}[t]{.48\textwidth}
    \centering
    \includegraphics[width=\linewidth]{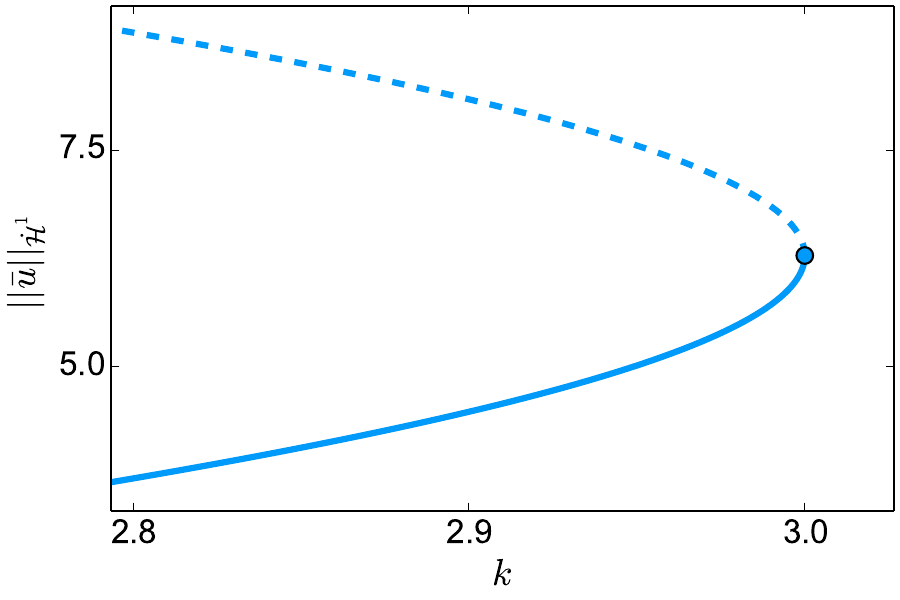}
    \caption{$\,$}\label{fig::fold1}
  \end{subfigure}
  \hfill
  \begin{subfigure}[t]{.48\textwidth}
    \centering
    \includegraphics[width=\linewidth]{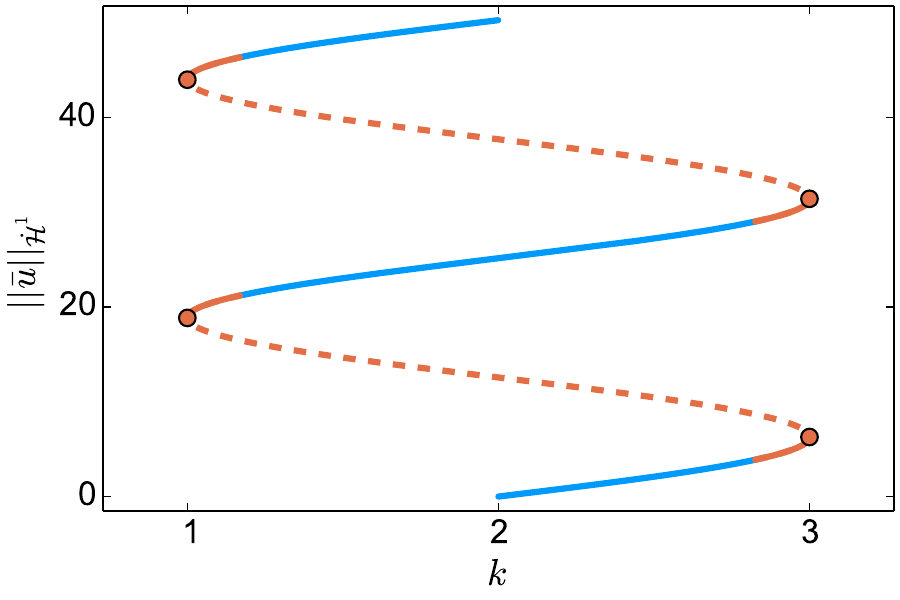}
    \caption{$\,$}\label{fig::fold2}
  \end{subfigure}
  \captionsetup{width=\linewidth}
  \caption{\subref{fig::fold1} An illustration of typical behaviour near quadratic fold point (depicted as a blue dot). Solid (respectively dashed) lines represent stable (resp. unstable) solutions. A change in stability at such points as shown in Proposition \ref{prop-eigen} is guaranteed by \eqref{nondeg1}, which ensures that the smallest eigenvalue passes through zero with nonzero `velocity'.\\ \subref{fig::fold2}: A schematic representation of a snaking curve with dots representing bifurcation points. The sets of solutions $\B_{\rm pos}$ and $\B_{\rm pt}$ defined in \eqref{B_st} are represented in blue and red, respectively. Note that $\B_{\rm pt}$ includes the entirety of the unstable segments, as well as bifurcation points and small parts of the stable segments.} \label{fig::schematic_fold}
\end{figure}

A schematic representation of the idea behind Definition \ref{deffp} is shown in Figure \ref{fig::fold1}.
As already discussed in the introduction, the fact that $\hat{u}_k \not\in\Hcc$ is key to \eqref{nondeg2} holding true, and suggests that a full bifurcation diagram is an infinite non-self-intersecting snaking curve \cite{TAYLOR201014}, consisting solely of regular and fold points as shown in Figure \ref{fig::fold2}. Our functional setup is well-suited to considering an arbitrary finite segment of it, so we begin with the following set of assumptions. {\cb We emphasise that all our subsequent results rely on the validity of these assumptions which are natural (see discussion below) but likely difficult to prove rigorously. Moreover, it is not guaranteed that Assumption 1 in particular is generic, but different potentials and loading geometries may indeed give rise to qualitatively different bifurcation diagrams.}
\begin{assumption}\label{ass1}
There exists a bifurcation diagram in the form of an injective continuous path $\B\,:\,[0,1] \to \Hcc\times \R$ given by
\begin{equation}\label{Bset}
\mathcal{B}(s):= (\bar{u}_s,\bar{k}_s),
\end{equation}
where ${\rm Im}(\B) \subset S$ (defined in \eqref{Sset}) is compact and for each $s \in [0,1]$, $\B(s)$ is either a regular point, as in \eqref{regp}, or a fold point, as in Definition \ref{deffp}. We further assume that there are finitely many fold points occuring at  $s \in\{b_1,\dots,b_M\} \subset (0,1)$. In particular, this implies that ${\rm Im}(\B)$ is a non--self--intersecting curve.
\end{assumption}

For future reference, if $f : \mathcal{B} \to X$, where $X$ is a Banach space, is differentiable, then we write $f'_s := \frac{d}{ds} f_s$.

\begin{assumption}\label{ass2-new}
There exists $c > 0$ such that for each $s \in [0,1]$ there exists a subspace $U_s$ of $\Hcc$ of codimension at most 1 for which it holds that
\begin{equation}\label{stab2-new}
\<{H_sv}{v} \geq c\,\|v\|_{\Hcc}^2
\end{equation}
for all $v \in U_s$.
\end{assumption}
The fact that a succession of fold points occurs is assumed to be an inherent feature of the lattice and the potential in place, much as the existence of a solution to a static dislocation problem is assumed in \cite{EOS2016}. Assumption \ref{ass2-new} ensures that each $\B(s) = (\bar{u}_s,\bar{k}_s)$ represents either a bifurcation point, a stable solution or an unstable solution which is an index--1 saddle point. This assumption is motivated by the fact that the anti-plane setup and lattice symmetry naturally binds the crack propagation to the $x_1$-axis, leaving little room for any more involved bifurcating behaviour. Moreover, this is also supported by numerical evidence presented in Section \ref{numerics}{\cb , in particular with Figure \ref{fig::numerics_convergence1} clearly exhibiting the snaking curve structure of the bifurcation diagram. We also refer to Section \ref{par-perbif} for a discussion about the periodicity of the bifurcation diagram which further justifies Assumptions \ref{ass1} and \ref{ass2-new}. Finally we note that in \cite{Li2013} a similar numerical evidence is presented for a vectorial Mode I fracture model posed on a triangular lattice under Lennard-Jones potential.} 

As will be shown in Proposition \ref{prop-eigen}, requiring that \eqref{nondeg1} holds ensures that a change in the stability of the solution occurs at each fold point. This implies that near bifurcation points and on the unstable segments the infimum of the spectrum of $H_s$ is an eigenvalue, which motivates the following decomposition of the parametrisation interval $[0,1]$: since we look at a finite segment of the full bifurcation diagram, we will assume for notational convenience that it starts on a stable segment and  the number of fold points $M$ lying in ${\rm Im}(\B)$ is even and define sets
\begin{equation}\label{S_st_unst}
\mathcal{I}_{{\rm pt}}:= \bigcup_{k=1}^{M/2}I_k \subset [0,1]\quad\text{and}\quad \mathcal{I}_{\rm pos} := [0,1]\setminus \mathcal{I}_{{\rm pt}},
\end{equation}
where $I_k := (b_{2k-1} - \xi, b_{2k} + \xi)$ with $\xi > 0$ small enough. The cases where $M$ is odd or we start on an unstable segment can be handled in an entirely analogous way. We refer to
\begin{equation}\label{B_st}
\mathcal{B}_{\rm pt} := \B(\mathcal{I}_{\rm pt})\quad\text{and}\quad\mathcal{B}_{\rm pos} := \B(\mathcal{I}_{\rm pos})
\end{equation}
as the collection of segments of the bifurcation diagram with $\sigma_p(H_s) \neq \emptyset$ (non--empty point spectrum) and $\sigma (H_s) \subset [c,\infty)$ (positive spectrum, $c$ from Assumption \ref{ass2-new}), respectively. We note that both the unstable segments and neighbourhoods of the bifurcation points belong to $\B_{\rm pt}$, thus the constant $c$ in Assumption \ref{ass2-new} can be chosen to be small enough so that
\begin{equation}\label{Us-stable}
s \in \mathcal{I}_{\rm pos} \implies U_s = \Hcc.
\end{equation}

We now establish some initial results about the model. First, a regularity result.
\begin{prop}[Regularity of the diagram]\label{prop-b}
The set ${\rm Im}(\mathcal{B})\subset \Hcc\times\R$ is a one--dimensional $C^{\alpha-1}$ manifold.
\end{prop}
{\cb This result will be proven in Section \ref{p-model} and in particular }entails that, without loss of generality, we may make the following assumption concerning the parametrisation $\B$.
\begin{assumption}\label{ass3}
 The function $\B:[0,1]\to\Hcc\times \R$ is a constant speed, $C^{\alpha-1}$ parametrisation of the manifold ${\rm Im}(\B)\subset \Hcc\times \R$.
\end{assumption}

We next {\cb prove} a result concerning the existence of linearly unstable directions and corresponding negative eigenvalues for some sections of the bifurcation diagram.

\begin{prop}[Existence of an eigen-pair]\label{prop-eigen}
Under Assumptions \ref{ass1}, \ref{ass2-new} \& \ref{ass3}, there exist $C^{\alpha-2}$ functions $\gamma\,:\, \mathcal{I}_{{\rm pt}} \to \Hcc$ and $\mu\,:\, \mathcal{I}_{{\rm pt}} \to \R$ such that
\begin{equation}\label{eigen-pair}
	H_s\gamma_s = \mu_s \Rie\gamma_s,
\end{equation}
where $H_s$ was defined in \eqref{def-hess} and $\Rie$ represents the Riesz mapping \cite{rudin}, i.e. an isometric isomorphism between $\Hcc$ and $(\Hcc)^*$, thus we can equivalently say that
\[
\<{H_s\gamma_s}{v} = \mu_s(\gamma_s,v)_{\Hcc}\qquad\text{for all }v \in \Hcc.
\]
Furthermore, for $j=1,\dots,M$, we have $\mu_{b_j} = 0$ with the corresponding eigenvector $\gamma_{b_j}$ introduced in Definition \ref{deffp} and also $\mu_{b_j}' \neq 0$, implying that a change of stability occurs at $s=b_j$.
\end{prop}
{\cb This will be proven in Section \ref{p-model}. }

We subsequently establish the following decay and regularity results for the atomistic core corrector, which rely on the precise characterisation of the lattice Green's function for the anti-plane crack geometry developed in \cite{2018-antiplanecrack}.

\begin{theorem}[Decay properties of solutions and eigenvectors]\label{thm:us}
For any $s \in [0,1]$ and $l \in \La$ with $|l|$ large enough it holds that for any $\delta >0$ the atomistic correction $\bar{u}_s$ satisfies
\begin{equation}\label{usdecay}
\big|\tilde{D}\bar{u}_s(l)\big| \leq C|l|^{-3/2+\delta}.
\end{equation}
If $s \in \mathcal{I}_{\rm pt}$, then the eigenvector $\gamma_s \in \Hcc$ from Proposition \ref{prop-eigen} satisfies
\begin{equation}\label{gammasdecay}
\big|\tilde{D}\gamma_s(l)\big| \leq C|l|^{-3/2+\delta}.
\end{equation}
In both cases $C$ is a generic constant independent of $s$.
\end{theorem}
{\cb As in the case of Theorem \ref{Eu-well-def}, we note that \eqref{usdecay} can be proven as in \cite{2018-antiplanecrack}, except for an extra dificulty arising from the fact that bonds across the crack are now included. The estimate in \eqref{gammasdecay} follows from a two-step argument that is similar in nature. Both these estimates will be proven in Section \ref{p-model}.}
{ \cb
\begin{remark}
The arbitrarily small $\delta > 0$ in Theorem \ref{thm:us} appears due to a technical limitation of the method employed in \cite{2018-antiplanecrack} to estimate the mixed second discrete derivative of the lattice Green's function in the anti-plane crack geometry. We expect the result to hold for $\delta = 0$ too, but this cannot be achieved with the current bootstrapping argument, which saturates at the known decay rate of the corresponding continuum Green's function. 
\end{remark}}
\subsection{Approximation}\label{sec:approx}
As numerical simulations are naturally restricted to a computational domain of finite size, we now consider and analyse a finite-dimensional scheme that approximates the solution path $\mathcal{B}$ defined in \eqref{Bset} and establish rigorous convergence results.

The starting point is a computational domain $\Omega_R$ with $B_R \cap \La \subset \Omega_R \subset \La$ (where $B_R$ is a ball of radius $R$ centred at the origin) and the boundary condition prescribed as $\hat{u}$ on $\La\setminus \Omega_R$. The approximation to \eqref{Sset} can thus be stated as a Galerkin approximation, that is we seek to characterise
\begin{equation}\label{Ssetapp}
S_R:= \big\{(u^R,k) \in {\cb \Hc^0_R} \times \R\,\big|\, \delta_u \E(u^R,k) = 0 \in (\Hc^0_R)^*\big\},
\end{equation}
where
\[
\mathcal{H}^0_R := \{v\,\colon\,\La\to\R \,|\,v = 0 \text{ in } \La\setminus\Omega_R\}.
\]
We prove the following.
\begin{theorem}\label{approx-thm-1}
Under  Assumptions \ref{ass1}, \ref{ass2-new} \& \ref{ass3}, there exists $R_0 >0$, such that for all $R \geq R_0$, there exists a $C^{\alpha-1}$ approximate bifurcation path $\B_R\,:\,[0,1] \to \Hc_R^0 \times \R$ given by
\[
\mathcal{B}_R(s):= (\bar{u}^R_s,\bar{k}^R_s),
\]
where ${\rm Im}(\B_R) \subset S_R$, such that for any $\beta > 0$
\begin{equation}\label{BR_conv_est}
\|\bar{u}_s^R - \bar{u}_s\|_{\Hcc} +
 	\big|\bar{k}_s^R - \bar{k}_s\big| \lesssim R^{-1/2+\beta}
\end{equation}
and
\begin{equation}\label{BR_norm_conv_est}
\big|\E(\bar{u}_s^R,\bar{k}_s^R) - \E(\bar{u}_s,\bar{k}_s)\big| \lesssim R^{-1+\beta},
\end{equation}
where $\mathcal{B}(s) = (\bar{u}_s,\bar{k}_s)$ as in \eqref{Bset}.
\end{theorem}
{\cb For the proof of this result, we refer to Section \ref{p-conv}.}

While the estimate in \eqref{BR_conv_est} appears to be almost sharp (our numerical results in Section~\ref{numerics} indicate that this estimate holds with $\beta=0$), more can be said about the approximation of the critical values of the stress intensity factor for which fold points occur.
\begin{theorem}\label{approx-thm-2}
For $R$ suffiently large, the approximate bifurcation path $\mathcal{B}_R$ from Theorem~\ref{approx-thm-1} contains $M$ fold points in the sense of Definition \ref{deffp} occuring at $s \in \{b_1^R,\dots,b_M^R\} \subset (0,1)$, and for each of these we have
\[
\big|\bar{k}^R_{b_j^R} - \bar{k}_{b_j}\big| \lesssim R^{-1+\beta}\quad\text{for any }\beta >0.
\]
\end{theorem}
{\cb For the proof, we again refer to Section \ref{p-conv}.}

This superconvergence result is well-known in bifurcation theory \cite{cliffe_spence_tavener_2000}, and is observed numerically in our tests in Section \ref{numerics}.
\subsection{Numerical investigation}\label{numerics}
In this section we present results of numerical tests that confirm the rate of decay of $|\tilde{D}\bar{u}_s|$ and $|\tilde{D}\gamma_s|$ established in Theorem \ref{thm:us}, as well as the convergence rates from Theorems \ref{approx-thm-1} and \ref{approx-thm-2}. The computational setup is similar to the one described in \cite[Section 3]{2017-bcscrew}, with $\La$ and $\Rc$ as specified in Section~\ref{setup} and the pair-potential given by
\begin{equation}\label{pot_num}
\phi(r) = \frac{1}{6}\big(1 - \exp(-3r^2)\big).
\end{equation}
We employ a pseudo-arclength numerical continuation scheme to approximate $\B_R$ \cite{Beyn_numericalcontinuation}. To compute equilibria we employ a standard Newton scheme, terminating at an $\ell^{\infty}$-residual of $10^{-8}$.

Theorem \ref{thm:us} suggests that  $\lvert \tilde{D} \bar{u}_s(l) \rvert \lesssim \lvert l \rvert^{-3/2}$ and $\lvert \tilde{D} \gamma_s(l) \rvert \lesssim \lvert l \rvert^{-3/2}$. This is verified in Figure \ref{fig::numerics_decay}.
Theorem \ref{approx-thm-1} suggests that in the supercell approximation of $\B$ in \eqref{Bset} we expect $\|\bar{u}_s^R - \bar{u}_s\|_{\Hcc} + |\bar{k}_s^R - \bar{k}_s| \sim \mathcal{O}(R^{-1/2})$, where $R$ is the size of the domain. To verify this numerically, we first compute $\B_R$ for $R=32,\dots,256$ via a pseudo--arclength continuation scheme. The results are shown in Figure \ref{fig::numerics_convergence1}, with stable segments plotted as solid lines and unstable segments as dashed lines. To measure the distance between the segments of the bifurcation diagram, we compute the Hausdorff distance \cite{RW98} with respect to $\|\cdot\|_{\Hcc}$-norm between the critical points on $\B_R$ (for $R=32,\dots,90.51$) and on $\B_{R_*}$, where $R_* = 256$. The result is shown in Figure \ref{fig::conv1}.

Finally, we test the superconvergence result for the bifurcation points from Theorem \ref{approx-thm-2}, which predicts that $|\bar{k}^R_{b_j^R} - \bar{k}_{b_j}| \sim \mathcal{O}(R^{-1})$. To this end we accelerate the convergence of the sequence $\{\bar{k}^R_{b_j^R}\}$ (for $R=32,\dots,256$) by employing Richardson extrapolation \cite{richardson}, thus giving us an approximate limit value for $\bar{k}_{b_j}$. The values of the approximate limits, as well as the $R^{-1}$ convergence is exhibited in Figure \ref{fig::conv2}.
\begin{figure}[!htb]
  \begin{subfigure}[t]{.48\textwidth}
    \centering
    \includegraphics[width=\linewidth]{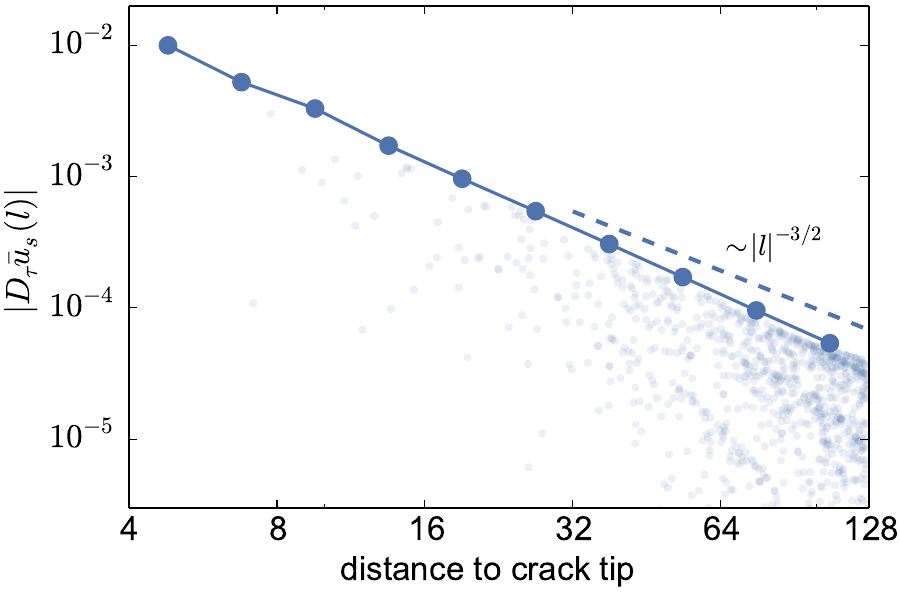}
  \end{subfigure}
  \hfill
  \begin{subfigure}[t]{.48\textwidth}
    \centering
    \includegraphics[width=\linewidth]{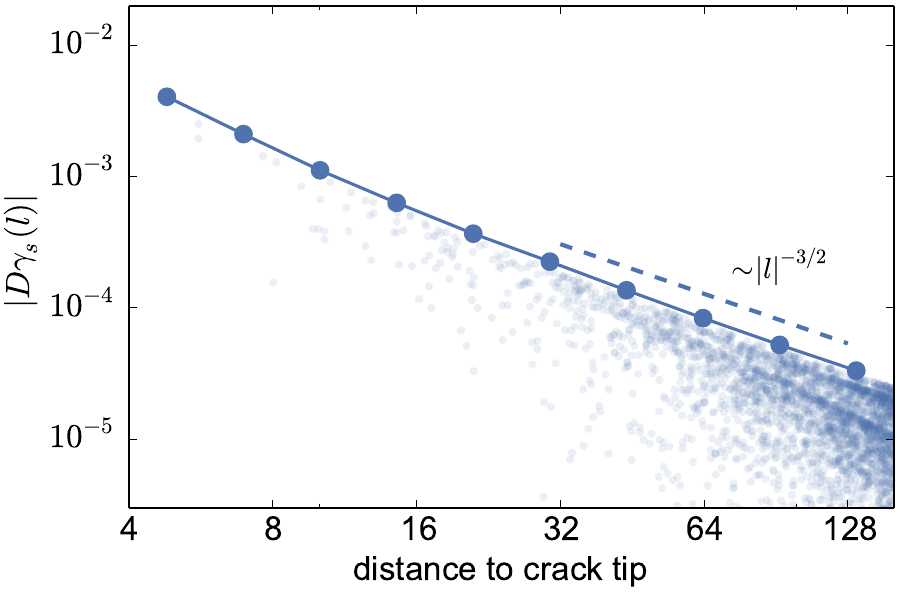}
  \end{subfigure}
  \captionsetup{width=\linewidth}
  \caption{ The decay of $D\bar{u}_s$ and $D\gamma_s$ { \cb for a domain with radius $R=256$}. Transparent dots denote data points $(|l|, |Du_s(l)|)$, solid curves their envelopes. We observe the expected rate of $|l|^{-3/2}$. {\cb The $x$-axis stopping near $|l| = 128$ ensures we do not observe any boundary effects.}   }
  \label{fig::numerics_decay}
\end{figure}

\begin{figure}[!htb]
  \begin{subfigure}[t]{.55\textwidth}
    \centering
    \includegraphics[width=\linewidth]{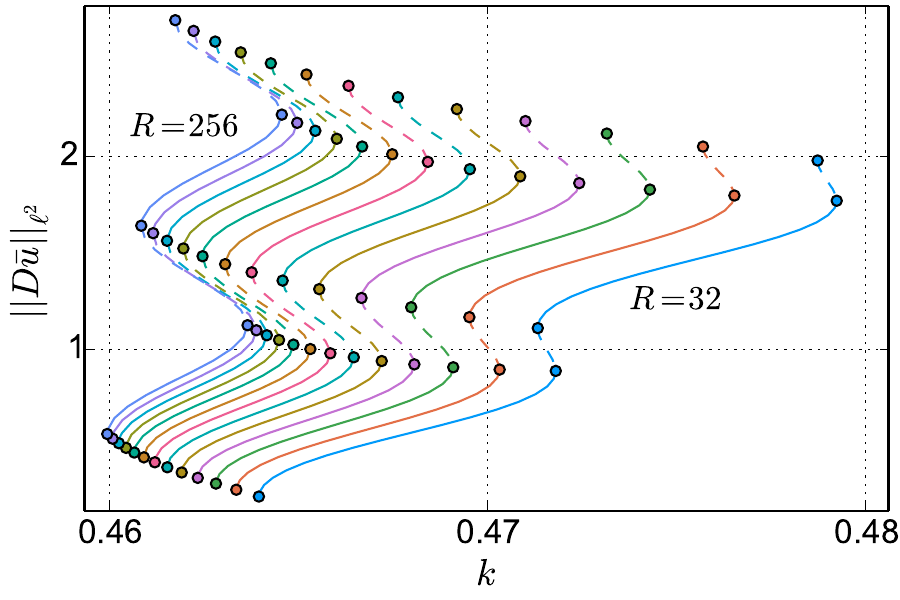}
  \end{subfigure}
  \captionsetup{width=\linewidth}
  \caption{The bifurcation paths $\B_R$ for $R=32,\dots,256$, that is for $R = 2^{n/4}$ for $n=20,\dots 32$. Solid lines denote stables segments, dashed lines unstable segments and dots the bifurcation points.
  } \label{fig::numerics_convergence1}
  \end{figure}
  \begin{figure}[!htb]
    \begin{subfigure}[t]{.39\textwidth}
    \centering
    \includegraphics[width=\linewidth]{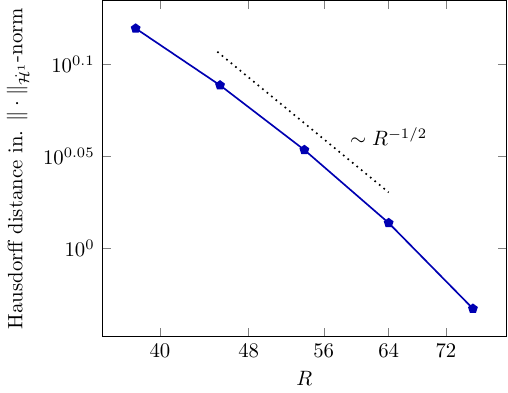}
    \caption{$\,$}\label{fig::conv1}
  \end{subfigure}
  \hfill
  \begin{subfigure}[t]{.59\textwidth}
    \centering
    \includegraphics[width=\linewidth]{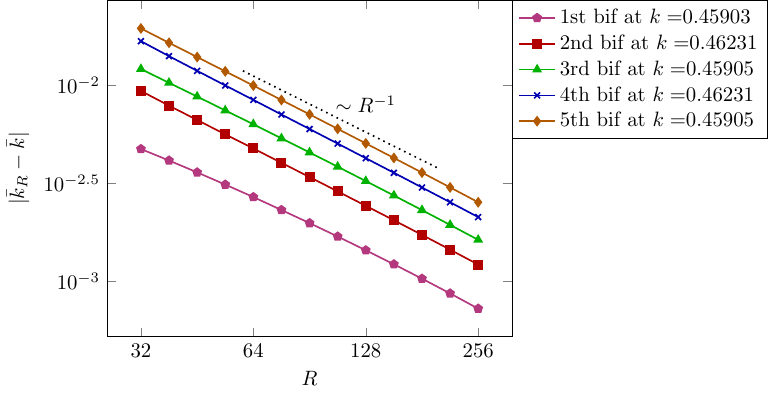}
    \caption{$\,$}\label{fig::conv2}
  \end{subfigure}
  \captionsetup{width=\linewidth}
  \caption{\subref{fig::conv1} The approximate rate of convergence of the supercell approximation of $\B$, measured by the Hausdorff distance with respect to $\|\cdot\|_{\Hcc}$-norm, compared against the domain with $R=256$.\\
  \subref{fig::conv2} The convergence rate of the values of stress intensity factor at which bifurcations occur. The approximate limit values as predicted by Richardson extrapolation are given in the legend entries. The fact that all unstable-to-stable (and separately stable-to-unstable) fold points occur at the same values indicates that in the limit the bifurcation path is exactly vertical.
  } \label{fig::numerics_convergence2}
\end{figure}
\begin{remark}
The pair-potential $\phi$ defined in \eqref{pot_num} does not satisfy the strong assumption of compact support of $\phi'$ introduced in \eqref{V-new-ass}, but has the slightly weaker property of exponential decay in first derivative. It thus emphasises the already discussed point that \eqref{pot_num} is by no means a necessary condition. {\cb It also emphasises the point that our results are of more practical applicability, despite the severe non-convexity of the energy landscape forcing us to introduce structural assumptions on the bifurcation diagram as opposed to proving them.}
\end{remark}

\section{Conclusion and discussion}\label{conclusion}
The results obtained here, in tandem with those of our previous paper \cite{2018-antiplanecrack}, introduce a mathematical framework in which a rigorous formulation and study of atomistic models of cracks and their propagation is possible. In particular, we have shown how the theory of atomistic modelling of defects developed in \cite{EOS2016,HO2014,2017-bcscrew} can be combined with classical results from bifurcation theory \cite{keller-1977,Brezzi1} to study this problem, and a key insight is the identification of the stress intensity factor as a suitable bifurcation parameter which allows us to explore the energy landscape. {\cb While our results are of conditional nature in that they rely on assumptions that are reasonably justified and numerically verified, our analysis nonetheless } sets earlier numerical work of \cite{Li2013,Li2014} into a rigorous framework and provides a comprehensive explanation as to why the bifurcation diagram is a snaking curve.

While further work is needed to extend our analytical results to more general models (and particularly to the case of other crack modes), from a numerical perspective several aspects of our theory are of universal applicability. We therefore conclude by pointing out a series of interesting conclusions which arise from our analysis.\\

{\cb \subsection{Applicability of the results to other crack models}
Due to reasons explored in the concluding section of \cite{2018-antiplanecrack}, at present we do not see an easy way of rigorously extending our results beyond a 2D model for scalar Mode III crack with nearest neighbours pair interactions on a square lattice. On a practical level, however, it can be numerically verified that our theory is entirely applicable to any 2D model for scalar Mode III crack with arbitrary finite range interactions under an arbitrary feasible interatomic potential and the resulting bifurcation diagrams is indeed a snaking curve.

For vectorial models of other modes of crack the numerical method described is still feasible, but depending on the potential employed, numerical tests indicate that one can expect a more complicated bifurcating behaviour, not least because of surface relaxation effects. The numerical evidence in \cite{Li2013} in particular indicates that the structure of the bifurcation diagram described in this paper does apply to some vectorial models. We hope to investigate this in greater detail in the future.}\\

\subsection{Periodicity of the bifurcation diagram}\label{par-perbif}  In an infinite lattice, shifting the crack tip by one lattice spacing results in a physically identical configuration. Therefore, it is reasonable to conjecture that $\B_R$ in the limit as $R \to \infty$ generates a bifurcation diagram in which the critical points exist for values of the SIF $k$ within a fixed finite interval of admissible values. In Section \ref{numerics} we have exploited the superconvergence result in Theorem \ref{approx-thm-2} to test this hypothesis numerically and the results summarised in Figure \ref{fig::numerics_convergence2} confirm this intuition, as the extrapolated limit values of SIF as $R\to\infty$ for every second bifurcation point are numerically identical, occuring at
\[
	k = 0.45903, \quad k = 0.46234, \quad k = 0.45905, \quad k = 0.46231, \quad k = 0.45905.
\]

A translation invariance in the critical points further implies that, if we denote the centre of the CLE predictor by
\begin{equation}\label{xlambda}
x_{\lambda} := (\lambda,0)
\end{equation}
for some $\lambda \in \Z$, and define $\E_{\lambda}\,:\,\Hcc\times \R \to \R$ by
\[
\E_{\lambda}(\bar u,k) = \sum_{m \in \La} V(D\hat{u}_{k}(m-x_{\lambda}) + D\bar u(m)) - V(D\hat{u}_{k}(m-x_{\lambda})),
\]
then assuming $\B(s)=(\bar u_s,\bar k_s)\in\Hcc\times \R$ is a parametrisation as described in Section~\ref{model}, we naturally have
\begin{equation}\label{tip-moving1}
\delta_u \E_{\lambda} (\bar{u}_s(\cdot-x_{\lambda}),\bar{k}_s) = 0.
\end{equation}
We further notice for any $s \in [0,1]$ the total displacement $y_s =  \hat{u}_{\bar{k}_s} + \bar{u}_s$ can be rewritten as $y_s(m) = \hat{u}_{\bar{k}_s }(m - x_{\lambda}) + w_{s,\lambda}(m)$, where $\lambda \in \Z$ and
\[
 w_{s,\lambda}(m) := \big(\hat{u}_{\bar{k}_s }(m) - \hat{u}_{\bar{k}_s }(m - x_{\lambda})\big) + \bar{u}_s(m).
\]
Crucially,
\begin{equation}\label{hatulambda}
|D\hat{u}_k(l) - D\hat{u}_k(l - x_{\lambda})| \lesssim |l|^{-3/2} \implies w_{s,\lambda} \in \Hcc
\end{equation}
and for any choice of $\lambda \in \Z$
\begin{equation}\label{tip-moving2}
\delta_u \E_{\lambda} (w_{s,\lambda},\bar{k}_s) = \delta_u \E(\bar{u}_s,\bar{k}_s) = 0.
\end{equation}
In other words, no matter which $x_\lambda$ we choose to centre the crack predictor $\hat u_k$ at, the same configuration $y_s=\hat{u}_{\bar{k}_s } + \bar{u}_s$ always remains an equilibrium and $w_{s,\lambda}$ exactly captures the resulting changes to the atomistic correction.

To be precise, let us fix some $s \in [0,1]$ thus giving us a pair $\mathcal{B}(s) = (\bar{u}_s,\bar{k}_s)$, let $K := \bar{k}_s$ and further consider
\[
I_K := \{ s \in [0,1]\,|\, \bar{k}_s = K\}.
\]
Equations \eqref{tip-moving1} and \eqref{tip-moving2} indicate that for any $s' \in I_K$ for which $\bar{u}_{s'}$ is a solution of the same type as $\bar{u}_s$ (either stable, or unstable or a bifurcation point), we can find a unique $\lambda \in \Z$ such that $\bar{u}_{s'} = w_{s,\lambda}$.

In particular we note that the above strongly suggests that in some cases one may be able to prove results about periodicity and boundedness in $k$ of the bifurcation diagram, which {\cb would also have the additional benefit of proving that Assumptions \ref{ass1} and \ref{ass2-new} hold true. The particular difficulty, however, lies in the fact that without these structural assumptions we cannot easily conclude that $\bar{u}_{s'} = w_{s,\lambda}$ for some unique $\lambda \in \Z$. This can potentially be proven under suitable (prohibitively restrictive) technical asssumptions on the potential, as explored for anti-plane screw dislocations in \cite{2014-dislift}.  We hope to investigate this in future work.}\\

\subsection{Interplay between the stress intensity factor and the domain size and its effect on lattice trapping.} The tilt of bifurcation diagrams seen in Figure \ref{fig::numerics_convergence1} indicates that the size of the domain heavily impacts the shape of the corresponding solution curve. Notably, each successive bond-breaking event has a different interval of admissible values of SIF associated to it and the corresponding unstable segments are much shorter than stable ones for small domain sizes. The fact that the influence of such finite-domain effects can still be observed for a fairly large $R$ can be explained by the very slow rate of convergence in Theorem \ref{approx-thm-2}.

In practice, one hopes to investigate crack propagation and associated energy barriers for a fixed value of SIF and subsequently compare it against other admissible choices of SIF to measure the strength of lattice trapping \cite{Thomson_1971,gumbsch_cannon_2000}, measured by the relative height of the energy barrier. Our work indicates that such investigations are particularly challenging due to the extent to which finite--size effects dominate, an effect we observe to be strong even in the simple model considered here. Only a very large choice of truncation radius $R$ ensures that the resulting solution paths are close to the periodic results one expects in the full lattice case. It may be possible to overcome such difficulties by prescribing a more accurate predictor describing the far--field behaviour, in line with the idea of development of solutions introduced in \cite{2017-bcscrew}: this is a clear direction for future investigation.\\

\subsection{Identification of the correct bifurcation parameter.} It is interesting to note that varying the intuitively natural bifurcation parameter $\lambda$ introduced in \eqref{xlambda} to reflect the crack tip at which the continuum prediction is centred in fact fails to capture the bifurcation phenomenon. This can be seen by considering  $\tilde{\E}\,:\,\Hcc\times\R \to \R$ given by
\[
\tilde{\E}(v,\lambda) = \sum_m V\big(D\hat{v}_{\lambda}(m) + Dv(m)\big) - V\big(D\hat{v}_{\lambda}(m)\big),
\]
where $\hat{v}_{\lambda}= \hat{u}_K(\cdot - x_{\lambda})$ for some fixed SIF $K > 0$. This fundamentally differs from the energy defined in \eqref{energy}, as in that case we have linear dependence of the displacement on $k$, $D\hat{u}_{k}(m) = kD\hat{u}_1(m)$, which in turn leads to a particular form of derivatives with respect to $k$; in particular, this ensures we have quadratic fold points.  In $\tilde{\E}$, however, the dependence is inside $\hat{v}_\lambda$, thus the crucial derivative with respect $\lambda$ is given by
\[
  \co{\delta^2_{v\lambda}\tilde{\E}(v,\lambda)[w,h] = h \sum_{m} \delta^2V(D\hat{v}_{\lambda}(m) + Dv(m))Df_{\lambda}(m) \cdot Dw(m),}
\]
where $f_{\lambda}(m) = \nabla \hat{u}_K(m_1-\lambda,m_2)\cdot e_1$. In this case, as in \eqref{hatulambda}, we can conclude $f_{\lambda} \in \Hcc$. This implies that a fold point cannot occur, as that would require that there exists $\gamma \in \Hcc$ such that
\begin{equation*}
\co{\<{\delta_{vv}^2 \tilde{\E}(v,\lambda)\gamma}{v} = \sum_{m} \delta^2 V\big(D\hat{v}_{\lambda}(m) + Dv(m)\big)D\gamma(m) \cdot Dw(m)= 0}
\end{equation*}
for all $w \in \Hcc$, which further implies that
\begin{align*}
\co{\delta^2_{v\lambda}\tilde{\E}(v,\lambda)[\gamma,1]} &=  \sum_{m} \delta^2 V\big(D\hat{v}_{\lambda}(m) + Dv(m)\big)Df_{\lambda}(m) \cdot D\gamma(m) \\
&= \sum_{m} \delta^2 V\big(D\hat{v}_{\lambda}(m) + Dv(m)\big)D\gamma(m) \cdot Df_{\lambda}(m) = 0,
\end{align*}
since $f_{\lambda} \in \Hcc$. This breaches the defining property of a fold point given in Definition~\ref{deffp}; in fact, it is not possible to drive a bifurcation in this way precisely because of the observation made in \eqref{hatulambda} and the resulting periodicity. \\

\subsection{Parameter-driven analysis for other models and defects} An overarching idea of this paper is that a careful analysis of a crucial parameter involved in the model can reveal the energy landscape of the problem, and this can be particularly fruitful in the study of defect migration. In the case of a crack, using the SIF as a driving parameter naturally generalises to more complex fracture models, though in general there may be multiple SIFs.

It would be interesting to undertake future study to see whether such an analysis is applicable more widely to other defects. In the particular case of dislocations, nucleation and motion have been studied in the atomistic context in a number of recent studies, including \cite{CHAUSSIDON2006,ADLGP14,GARG2015,ADLGP17,H17}. Since these defects are the carriers of plastic deformation, the study of their mobility is important, and natural candidates for the parameters in this case are the shear modulus and externally--applied stress \cite{landau1989theory}.
\\
\newpage


\section{Proofs}\label{proofs}
\subsection{Preliminaries}\label{sec:pre}
Our approach is based on two classic results from bifurcation analysis on Banach spaces, cf. \cite{cliffe_spence_tavener_2000}, which we state in this section for convenience.
The first result is known as 'ABCD Lemma' and is adapted from \cite{keller-1977}.
\begin{lemma}[ABCD Lemma]\label{abcd}
Let $H$ be a Hilbert space with the dual $H^*$ and consider the linear operator $M\,:\,H\times \R \to H^* \times \R$ of the form
\[
M:= \begin{bmatrix}
    A       & b \\
    (c,\cdot)_H & d
\end{bmatrix},
\]
where $A\,:\,H\to H^*$ is self-adjoint in the sense that $\<{Av}{w} = \<{Aw}{v}$ for all $v,w\in H$, $b \in H^*\setminus\{0\}$, $c \in H\setminus\{0\}$ and $d \in \R$. Then
\begin{enumerate}[label=(\roman*)]
\item if $A$ is an isomorphism from $H$ to $H^*$, then $M$ is an isomorphism between $H\times \R$ and $H^*\times \R$ if and only if $d - (c,A^{-1}b)_H \neq 0$; and
\item if ${\rm dim}\,{\rm Ker} (A) = {\rm codim}\,{\rm Range}(A) = 1$ with ${\rm Ker} (A) = {\rm span}\{\gamma\}$, then $M$ is an isomorphism if and only if $\<{b}{\gamma} \neq 0$ and $(c,\gamma)_H \neq 0$.
\end{enumerate}
\end{lemma}
To state the second result we introduce the following setup: let $X$, $Y$ and $Z$ be real Banach spaces and $F\in C^k(U\times Y; Z)$ for some $k \geq 1$, where $U$ is a bounded open subset of $X$.  The total derivative of $F$ at $(x,y)\in X\times Y$ is denoted $DF(x,y) \in \mathcal{L}(X\times Y,Z)$, with partial derivatives denoted $D_x F(x,y) \in \mathcal{L}(X,Z)$ and $D_y F(x,y) \in \mathcal{L}(Y,Z)$. We now state a version of \cite[Theorem 1]{Brezzi1} tailored to our setting.
\begin{theorem}\label{thm:brezzi1}
Suppose a function $y\,:\,U \to Y$ is Lipschitz continuous with Lipschitz constant $c_2$, and there exist constants $c_0$ and $c_1$ and a monotonically increasing function $L_1:\R\to\R$ such that the following hypotheses are satisfied:
\begin{enumerate}[label=(\roman*)]
\item for any $x_0 \in U$, $D_yF\big(x_0,y(x_0)\big)$ is an isomorphism of $Y$ onto $Z$ with
\begin{equation}\label{c0}
\sup_{x_0 \in U}\big\|D_y F\big(x_0,y(x_0)\big)^{-1}\big\| \leq c_0;
\end{equation}
\item we have the uniform bound
\begin{equation}\label{c1}
\sup_{x_0 \in U} \big\|D_x F\big(x_0,y(x_0)\big)\big\| \leq c_1;
\end{equation}
\item for any $x_0 \in U$ and all $(x,y)$ satisfying $\|x-x_0\| + \|y-y(x_0)\| \leq \xi$, we have
\[
\|DF(x,y) - DF(x_0,y(x_0)\| \leq L_1(\xi)\big(\|x-x_0\| + \|y-y(x_0)\|\big).
\]
\end{enumerate}
It then follows that there exist constants $a,d >0$ depending only on $c_0,c_1,c_2$ and $L_1$ so that whenever it holds that
\[
\sup_{x_0 \in U} \|F(x_0,y(x_0)\big)\| \leq d,
\]
then there exists a unique function $g$ defined from $\bigcup_{x_0 \in U} B(x_0,a)$ (union of balls centered at $x_0$ of radius $a$) to $Y$ such that
\[
F\big(x,g(x)\big) = 0.
\]
Moreover, $g$ is a $C^k$ function on its domain of definition, and for all $x_0 \in U$ and all $x \in B(x_0,a)$
\begin{equation}\label{brezzi-ineq}
\|g(x) - y(x_0)\| \leq K_0\Big(\|x-x_0\| + \big\|F\big(x_0,y(x_0)\big)\big\|\Big),
\end{equation}
where $K_0 > 0$ depends only on the constants $c_0$ and $c_1$.
\end{theorem}

For future reference we note that a suitable choice $a$ is given by
\begin{equation}\label{choiceofa}
	a = \min\left\{\frac{\hat{a}}{2},\frac{b}{2c_2}\right\} - \epsilon
\end{equation}
where $\hat{a} = \frac{b}{4M}$, $b$ is such that $b L_1(b) \leq \frac{1}{2M}$, $M = \max\{c_0,1+c_0c_1\}$ and $\epsilon > 0$ is {sufficiently small to ensure that $a$ is positive.}

\subsection{Proofs about the model}\label{p-model}
{\cb We begin with a technical lemma that is required to prove Theorem \ref{Eu-well-def}. 
\begin{lemma}\label{loglemma}
 If $v \in \Hcc$, then, for any $l \in \Gamma_{\pm}$,
 \begin{equation}\label{Prop12_42}
 |v(l)| \lesssim ||v||_{\Hcc}(1+\log|l|).
\end{equation} 
\end{lemma}
\begin{proof}
The argument in \cite[Proposition 12(ii)]{2012-lattint} proves the result for the case without a crack present. In that setting, the proof follows directly from  \cite[Theorem 2.2]{2012-ARXIV-ellRd}. In a crack geometry the constructions has to avoid crossing the crack $\Gamma_0$, hence we modify the argument. We distinguish two cases, depending on whether $l \in \Gamma_+$ or $l \in \Gamma_-$. We recall that $\hat{x} = \left(\tfrac{1}{2},\tfrac{1}{2}\right)$ and that by definition $v \in \Hcc \implies v(\hat{x}) = 0$. 
 
\textit{Case 1:}  Let $l \in \Gamma_+$, which implies that $(l-\hat{x})\cdot e_1 = 0$.  We consider a sequence of squares $(Q_i)_{i=0}^N \subset \R^2\setminus\left(\mathcal{D}_{\Gamma}\cup \Gamma_0\right)$, where $\mathcal{D}_{\Gamma}$ denotes the continuum region enclosed by $\Gamma_+$ and $\Gamma_-$. The squares are aligned in the direction $e_2$, which is possible due to the assumption on $l$, and defined as follows. $Q_0$ and $Q_N$ are unit squares corresponding to sites $\hat{x}$ and $l$, defined in such a way that $\hat{x}$ (respectively $l$) is the midpoint of the side of $Q_0$ (resp. $Q_N$) which borders $\Gamma_+$. The squares $Q_1,\dots,Q_{N-1}$ are defined to fill the space between $\hat{x}$ and $l$ in such a way that they have disjoint interiors and are such that their side-lengths differ by at most a factor of 2, with one side of the smaller square contained in one side of the larger square. It is easy to see that there is at most 
\[
N \lesssim (2 + \log|l-\hat{x}|) \lesssim 1 + \log|l|
\]
squares in the sequence. See Figure \ref{fig::crackfinal}. 
\begin{figure}[!htb]
 \centering
 \includegraphics[width=0.9\linewidth]{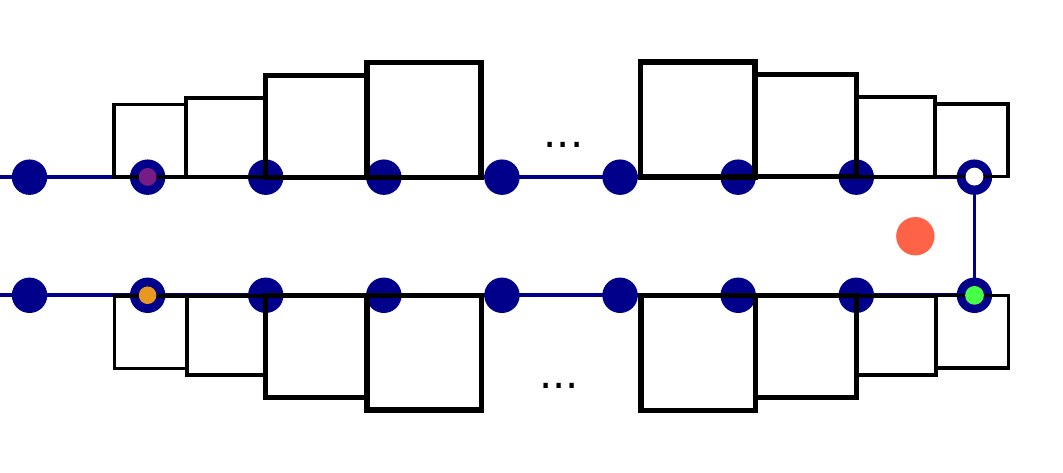}
  \caption{{\cb An example of construction of squares. The white dot represents $\hat{x}$, the green dot is $\hat{a}$, the purple dot a lattice site in Case 1 of Lemma \ref{loglemma} and the orange dot is a lattice site in Case 2.}}\label{fig::crackfinal}
\end{figure}

For any two neighbouring squares $Q_j,\,Q_{j+1}$ it follows from a special case of \cite[Lemma 2]{John1} that 
 \begin{equation}\label{apparg1}
 |(v)_{Q_{j+1}} - (v)_{Q_{j}}| \lesssim \|\nabla Iv\|_{L^2(\R^2\setminus\Gamma_0)} \lesssim \|v\|_{\Hcc},  
 \end{equation}
where
\[
 (v)_{Q_j} := \frac{1}{|Q_j|}\int_{Q_j}Iv(x)\,dx
\]
and $I$ denotes the crack domain P1 interpolation operator employed  in \cite[Theorem 2.3]{2018-antiplanecrack}. The final inequality in \eqref{apparg1} follows from the fact that both components of $\nabla I v$ are piecewise constant and each corresponds to $D_{\rho}v(l)$ for some bond $(m,m+\rho)$. 

As a result
\begin{align}\label{apparg4}
 |(u)_{Q_N} - (u)_{Q_0}| &\leq \sum_{j=1}^N|(u)_{Q_{j}} - (u)_{Q_{j-1}}|\\
 &\lesssim \sum_{j=1}^N \|\nabla Iv\|_{L^2(\R^2\setminus\Gamma_0)} = N\|v\|_{\Hcc} \lesssim (2+\log|l|)\|v\|_{\Hcc} \nonumber.
\end{align}
 
Furthermore, it is naturally true that 
\begin{equation}\label{apparg2}
|(v)_{Q_0}-v(\hat{x})| \leq \|\nabla v\|_{L^{\infty}(Q_0)}\quad\text{and}\; |v(l) - (v)_{Q_N}| \leq \|\nabla v\|_{L^{\infty}(Q_N)}
\end{equation}
and since on each $Q_i$, the piecewise linear interpolant $v$ belongs to a finite-dimensional space, we obtain
\begin{equation}\label{apparg3}
\|\nabla Iv \|_{L^{\infty}(Q_i)} \lesssim \|\nabla Iv \|_{L^2(Q_i)} \lesssim \|\nabla v\|_{L^2(\R^2\setminus\Gamma_0)} \lesssim \|v\|_{\Hcc}, 
\end{equation}
where the first inequality follows from the equivalence of norms for finite-dimensional spaces and the second from extending the domain from $Q_i$ to the whole of $\R^2\setminus\Gamma_0$. 

With (\ref{apparg4}) and (\ref{apparg2})-(\ref{apparg3}) in hand, we obtain
\begin{align}\label{apparg5}
 |v(l)| = |v(l) - v(\hat{x})| &\leq |v(l) - (v)_{Q_N}| + |(v)_{Q_N} - (v)_{Q_0}| + |(v)_{Q_0} - v(\hat{x})| \\
 &\lesssim (1+\log|l|)\|v\|_{\Hcc}\nonumber,
 \end{align}
which concludes the proof for $l \in \Gamma_+$.

\textit{Case 2:} Let $l \in \Gamma_-$. The fact that $l$ is on the other side of the crack relative to $\hat{x}$ deems the previous argument invalid, as we can no longer define the sequence of squares aligned with $\hat{x}$ and $l$ which will be a subset of $\R^2\setminus\left(\mathcal{D}_{\Gamma}\cup\Gamma_0\right)$. Thus we first 'jump' to the other side, that is we $\hat{a} = \left(\tfrac{1}{2},-\tfrac{1}{2}\right)$ and conclude that
\begin{align*}
 |v(l) - v(\hat{x})| &\leq |v(l) - v(\hat{a})| + |v(\hat{a}) - v(\hat{x})|\\
 &\lesssim \|v\|_{\Hcc}((1+\log|l-\hat{a}|))\\
 & \lesssim ||v||_{\Hcc}(1+\log|l|),
\end{align*}
where the second inequality follows from applying (\ref{apparg5}) to a sequence of squares between $l$ and $\hat{a}$ (in fact there are only two such squares) and the fact that a bound on $|v(\hat{a})-v(\hat{x})|$ can be incorporated into the general form . 
\end{proof}
We now show that the model is well-defined}, with the particular emphasis on two new elements of the analysis that are distinct from previous arguments of this kind, e.g. \cite{EOS2016,2018-antiplanecrack}.
\begin{proof}[Proof of Theorem \ref{Eu-well-def}]
We can decompose the energy into a bulk part and a crack surface part by writing
\[
\E(u,k) = \E_{\rm bulk}(u,k) + \E_{\Gamma}(u,k),
\]
where
\[
\E_{\rm bulk}(u,k) := \sum_{m\in\La}\sum_{\rho\in \tilde{\Rc}(m)} \phi\big(D_{\rho}\hat{u}_{k}(m) + D_{\rho}u(m)\big) - \phi\big(D_{\rho}\hat{u}_{k}(m)\big),
\]
\begin{equation*}
\E_{\Gamma}(u,k) := \sum_{m\in\Gamma_{\pm}}\sum_{\rho\in \Rc\setminus\tilde{\Rc}(m)} \phi\big(D_{\rho}\hat{u}_{k}(m) + D_{\rho}u(m)\big) - \phi\big(D_{\rho}\hat{u}_{k}(m)\big).
\end{equation*}
We notice that $\E_{\rm bulk}$ excludes the bonds across the crack and thus is well-defined on $\Hcc\times \R$, as shown in the first part of Theorem 2.3 in \cite{2018-antiplanecrack}.

To establish the same for the $\E_{\Gamma}$, we note that we have the symmetry
\begin{equation*}
\hat{u}_{-k}(l) = \hat{u}_{k}(l_1,-l_2);
\end{equation*}
using this observation, for $m\in\Gamma_{\pm}$ and $\rho \in\Rc\setminus\tilde{\Rc}(m)$ we have
\begin{equation}\label{uhat-across-crack}
|D_{\rho}\hat{u}_k(m)| = |-2\hat{u}_k(m)| \sim |m|^{1/2}
	\quad \text{as $|m| \to \infty$}.
\end{equation}
Furthermore, {\cb it follows from Lemma \ref{loglemma} that }$u \in \Hcc$ implies that $|Du(m)| \lesssim \log|m|$, which in particular implies that
\[
	|D_{\rho}\hat{u}_{k}(m) + D_{\rho}u(m)| \geq C_0 |m|^{1/2} - C_1 \log|m|,
\]
for suitable constants $C_0, C_1$. Therefore, using assumption \eqref{V-new-ass}, it follows that for any $m\in\Gamma_\pm$ with $|m|$ sufficiently large,
  \[
    \phi\big(D_{\rho}\hat{u}_{k}(m)+ D_{\rho}u(m)\big)-\phi\big(D_{\rho}\hat{u}_{k}(m)\big) =0;
  \]
this entails that for each $u \in \Hcc$ we effectively only sum over a finite domain, and implies that $\E_{\Gamma}$ is indeed well-defined over $\Hcc \times \R$.

The differentiability properties of the functional follow from a standard argument, see \cite{2012-ARMA-cb}; here we simply provide formulae for derivatives of relevance to our subsequent arguments. In particular, we have
\[
\<{\delta_k\E(u,k)}{\lambda} = \sum_{m\in\La} (\nabla V(D\hat{u}_k(m)+Du(m))-\nabla V(D\hat{u}_k(m)) \cdot (\lambda D\hat{u}_k(m)),
\]
\[
\<{\delta_u\E(u,k)}{v} = \sum_{m\in\La} \nabla V(D\hat{u}_k(m)+Du(m)) \cdot Dv(m),
\]
\begin{equation}\label{Euk}
\co{\delta^2_{uk}\E(u,k)[v, \lambda]} = \sum_{m\in\La} \nabla^2 V(D\hat{u}_k(m)+Du(m))[\lambda D\hat{u}_k(m)] \cdot Dv(m),
\end{equation}
\[
\co{\<{\delta^2_{uu}\E(u,k)v}{w}} = \sum_{m\in\La} \nabla^2 V(D\hat{u}_k(m)+Du(m))Dv(m) \cdot Dw(m),
\]
\begin{equation}\label{Euuu}
\co{\delta^3_{uuu}\E(u,k)[v,w,z]} = \sum_{m\in\La} \nabla^3 V(D\hat{u}_k(m)+Du(m))[Dv(m),Dw(m), Dz(m)].\qedhere
\end{equation}
\co{As stated, the foregoing expressions are valid for $v \in \Hc^{\rm c}$. To define them for $v \in \Hcc$ one requires an extension argument relying on showing that $\delta_u\E(0) \in (\Hcc)^{*}$, which is proven in \cite[Theorem 2.3]{2018-antiplanecrack}, and analogous results for the remaining terms.}
\end{proof}
We now turn to the analysis of the bifurcation path $\mathcal{B}$.
\begin{proof}[Proof of Proposition \ref{prop-b}]
  This is a standard result and follows from the fact that $\E$ is a $C^{\alpha}$ functional, so we may apply the local uniqueness of the function $g$ whose existence was asserted in Theorem \ref{thm:brezzi1}. We therefore only outline the proof. Define sets corresponding to neighbourhoods of fold points
\[
\mathcal{I}_{\rm f} := \bigcup_{i=1}^M (b_i-\xi,b_i+\xi)\quad\text{ and }\quad \B_{\rm f} := \B(\mathcal{I}_{\rm f}),
\]
Since ${\rm Im}(\B)$ is compact, then so is ${\rm Im}(\B)\setminus \B_{\rm f}$, thus the latter can be covered with a finite collection of neighbourhoods of points $\{(\bar{u}_{s_i},\bar{k}_{s_i})\}_{i=1,\dots,N}$. It will be shown in the proof of Theorem \ref{approx-thm-1} that $H_{s_i} = \delta^2_{uu}\E(\bar{u}_{s_i},\bar{k}_{s_i})$ at each such point is an isomorphism, thus rendering Theorem \ref{thm:brezzi1} applicable to $\delta_u\E(\bar{u}_{s_i},\bar{k}_{s_i})$, giving us a locally unique $C^{\alpha-1}$ graph of critical points $k\mapsto u(k)$, which by its uniqueness together with injectivity of $\B$ has to coincide with $(\bar{u}_s,\bar{k}_s)$, thus ${\rm Im}(\B)\setminus \B_{\rm f}$ is a piecewise $C^{\alpha-1}$ manifold. To establish the same in $\B_{\rm f}$, for each fold point $b_i$, one considers an extended system  $\tilde{F}\,:\,(\Hcc\times\R)\times\R$ given by $\tilde{F}(u,k,t) = \big(\delta_{u}\E(u,k),(u-\bar{u}_{b_i},\gamma_{b_i})_{\Hcc} - t\big)$, where $\gamma_{b_i}$ was introduced in Definition \ref{deffp}. The ABCD Lemma is applicable to this extended system evaluated at $(\bar{u}_{b_i},\bar{k}_{b_i},b_i)$, thus ensuring that Theorem \ref{thm:brezzi1} is also applicable, giving us a locally unique $C^{\alpha-1}$ graph $t\mapsto (u(t),k(t))$, where in particular $k(0) = \bar{k}_{b_i}$. Again, due to uniqueness this must coincide with $(\bar{u}_s,\bar{k}_s)$, and hence this finishes the argument.
%
\end{proof}

Likewise, the existence of an eigen-pair can be established.
\begin{proof}[Proof of Proposition \ref{prop-eigen}]
In what follows we always consider a system $G\,:\, B \times Y \to Z$ where $B \subset [0,1]$, $Y = \Hcc \times \R$ and $Z = (\Hcc)^* \times \R$, given by
\begin{equation}\label{Gext}
\co{G(s,\gamma,\mu):= (H_s\gamma - \mu \Rie\gamma, (c,\gamma)_{\Hcc} - 1),}
\end{equation}
where $c \in \Hcc$ will be chosen appropriately. We consider two subsets of $\B_{pt}$ seperately.
\medskip

\noindent
Throughout this proof we endow the product spaces with their canonical norms,
for example, $\|(u,k)\|_{\Hcc \times \R} = \|u\|_{\Hcc} + |k|$.

{\it (a) Vicinity of a bifurcation point:} We let $c = \gamma_{b_i}$ {\cb and, in order to simplify the notation for derivatives of $G$, we introduce $y: = (\gamma,\mu)$} and observe that
\[
{\cb D_y G(b_i,\gamma_{b_i},\mu_{b_i}) := D_y(b_i,y)_{y=(\gamma_{b_i},\mu_{b_i})} }= \begin{pmatrix}
  H_{b_i} - \mu_{b_i} \Rie & -\gamma_{b_i} \\
  (c,\cdot)_{\Hcc} & 0
 \end{pmatrix}.
\]
Thus, Lemma \ref{abcd} and Theorem \ref{thm:brezzi1} together imply that, for $s \in (b_i-\xi,b_i+\xi)$, where $\xi > 0$ is small enough (cf. \eqref{S_st_unst}), there exists an eigen-pair $(\mu_s, \gamma_s)$. To show that $ \mu_s' := \frac{{\rm d}\mu_s}{{\rm d} s} \neq 0$ at $s = b_i$, we differentiate both sides of \eqref{eigen-pair} with respect to $s$ to obtain
\begin{equation}\label{diffs}
\delta_{uuu}^3 \E(\bar{u}_s,\bar{k}_s)[\bar{u}_s',\gamma_s] + \delta_{uuk}^3 \E(\bar{u}_s,\bar{k}_s)[\bar{u}_s,\bar{k}_s'] + H_s[\gamma_s'] = \mu_s' \Rie[\gamma_s] + \mu_s\Rie[\gamma_s'].
\end{equation}
By definition, at a fold point we have $\bar{k}_{b_i}' = 0$ and since we can differentiate $\delta_u\E(\bar{u}_s,\bar{k}_s) = 0$ with respect to $s$ to get $H_s \bar{u}_s' + \bar{k}_s' b_s = 0$, it follows that $\bar{u}_{b_i}' = \alpha\gamma_{b_i}$ for some $\alpha \neq 0$ (constant speed of parametrisation). Testing \eqref{diffs} at $s = b_i$ with $\gamma_{b_i}$ and simplifying, we obtain
\begin{equation}\label{munonzero}
\mu_{b_i}' = \alpha\<{\delta^3_{u u u} \E(\bar{u}_{b_i},\bar{k}_{b_i}) [\gamma_{b_i}, \gamma_{b_i}]}{\gamma_{b_i}} \neq 0,
\end{equation}
which is nonzero by Assumption \eqref{nondeg1}. This completes case (a).\medskip

{\it (b) Unstable segment away from bifurcations: } We assume without loss of generality that at the bifurcation point $b_i$ we switch from a stable segment to an unstable segment. The result in (a) establishes existence of an eigenvector for $s \in (b_i-\xi,b_i+\xi)$, so we let $t_1 := b_i+\xi-\epsilon$, where $0 < \epsilon < \xi$ and thus  are able to set $c = \gamma_{t_1}$ in \eqref{Gext}. A subsequent application of Theorem~\ref{thm:brezzi1} to system $G$ with this newly chosen $c$ yields existence of a new interval $(t_1 - \xi_1, t_1 + \xi_1) \subset \mathcal{I}_{\rm pt}$ for which the premise of the theorem is true. This procedure can be iterated, for example by incrementing $t_2 := t_1 + \xi_{1}/2$ and repeating the argument. To cover the entire unstable segment in this way we need to bound $\xi_j$ from below, independently of $t_j$.

Due to \eqref{munonzero} we know that $\mu_{t} < 0$, which implies that the subspace $U_{t}$ from Assumption \ref{ass2-new} can be characterised as
\begin{equation}\label{Us-unstable}
U_{t} = \{v \in \Hcc\,|\, (v,\gamma_{t})_{\Hcc} = 0\}.
\end{equation}
With this in hand we consider any $(u,k) \in \Hcc \times \R$, decompose $u$ as $u = \alpha \gamma_{t} + v$, where $\alpha \in \R$ and $v \in U_{t}$, and aim to uniformly bound
\[
\co{\|D_y G(t,\gamma_{t},\mu_{t})[u,k]\| = \big\|H_{t}u - \mu_{t}\Rie u - k\Rie \gamma_{t} \big\| + \big|(\gamma_{t},u)_{\Hcc}\big|}
\]
from below.

To do so, we observe that
\begin{align}
		\label{eigeq1}
\<{(H_t - \mu_t\Rie)u - k \Rie\gamma_{t}}{\textstyle \frac{v}{\|v\|_{\Hcc}}}
	&= \<{(H_t - \mu_t\Rie)v}{\textstyle \frac{v}{\|v\|_{\Hcc}}} \geq (c-\mu_t)\|v\|_{\Hcc}, \\
	\label{eigeq2}
\<{(H_t - \mu_t\Rie)u - k \Rie\gamma_{t}}{-\gamma_{t}} &= \<{(H_t - \mu_t\Rie)\alpha\gamma_{t}-k\Rie \gamma_{t}}{-\gamma_{t}} = k
\end{align}
Together, \eqref{eigeq1}, \eqref{eigeq2} imply that
\begin{align}
	\notag
	\|H_{t}u - \mu_{t}\Rie u - k\Rie \gamma_{t}\|
	&= \sup_{\substack{\tilde{v} \in \Hcc \\ \|\tilde{v}\| = 1}}|\<{H_{t}u - \mu_{t}\Rie u - k\Rie \gamma_{t}}{\tilde{v}}| \\
	\notag
	&\geq \max\big( |c-\mu_t| \|v\|_{\Hcc}, |k| \big) \\
	\label{eq:eigeq12}
	&\geq {\textstyle \frac12} \min\big(|c-\mu_t|, 1\big)
		\big(\|v\|_{\Hcc} + |k|\big).
\end{align}
Moreover, we trivially have
\begin{equation}\label{eigeq3}
	|(\gamma_{t},u)_{\Hcc}| = |\alpha| = \|\alpha \gamma_t\|_{\Hcc}.
\end{equation}

Let $\tilde{c}_0(s)^{-1} =  \min\{\frac{1}{2}(c-\mu_{t}),\frac{1}{2}\} > \min\{\frac{1}{2}(c),\frac{1}{2}\} =: c_0^{-1}$, then
combining \eqref{eq:eigeq12} and \eqref{eigeq3} yields
\[
	\|D_y G(t,\gamma_{t},\mu_{t})[u,k]\|
		\geq
	\tilde{c}_0(t)^{-1}\|(u,k)\|.
\]

In a similar vein, we observe that
\[
\|D_sG(\gamma_{t},\mu_{t},t)\| \leq \|\delta_{uuu}^3 \E(\bar{u}_t,\bar{k}_t)\| + \|\delta_{uuk}^3 \E(\bar{u}_t,\bar{k}_t)\||\bar{k}_t'| =: \tilde{c}_1(t) \leq c_1,
\]
where $c_1 := \max_{t \in \mathcal{I}_{pt}} \tilde{c}_1(t)$, which is guaranteed to exist due to $\E$ being a $C^{\alpha}$ functional and that $\B$ is a $C^{\alpha-1}$ function of $s$, where $\alpha \geq 5$ by assumption.

It is also evident that Condition (iii) from Theorem \ref{thm:brezzi1} is satisfied with
\[
L_1(\xi) := \sup_{(s,y)\in S(\tilde{s},\gamma_{\tilde{s}},\mu_{\tilde{s}},\xi)} \|D^2G(s,y)\|,
\]
where $S(\tilde{s},\gamma_{\tilde{s}},\mu_{\tilde{s}},\xi) = \{ (s,\gamma,\mu) \in \R \times \Hcc \times \R \, \colon \, |s - \tilde{s}| + \|\gamma - \gamma_{\tilde{s}}\| + |\mu - \mu_{\tilde{s}}| \leq \xi \}$.

Guided by \eqref{choiceofa}, we define $M:= \max\{c_0, 1+c_0c_1\}$, choose $b$ such that $b L_1(b) \geq \frac{1}{2M}$, let $\hat{a} = \frac{b}{4M}$ and recall from \eqref{choiceofa} that $\xi = \min\left\{\frac{\hat{a}}{2},\frac{b}{2c_2}\right\} - \epsilon$ with sufficiently small $\epsilon >0$ is an admissible choice for $\xi$ which is in particular independent of $t$.

This completes the proof of part (b).


{\it (c) Regularity: } It remains to establish the $C^{\alpha-2}$-regularity of $s \mapsto (\mu_s, \gamma_s), s \in \mathcal{I}_{\rm  pt}$. To that end,  we note that $G$ is a smooth function of $\gamma$ and $\mu$ and $C^{\alpha-2}$ function with respect to $s$, thus uniqueness and regularity parts of Theorem \ref{thm:brezzi1} immediately imply that $\gamma$ and $\mu$ are $C^{\alpha-2}(\mathcal{I}_{\rm pt})$ functions.
\end{proof}

We are now in a position to prove the spatial regularity of $\bar{u}_s$ and $\gamma_s$.

\begin{proof}[Proof of Theorem \ref{thm:us}]
We begin by defining {\cb 
\[
v(m):= D_{2\tau}\G(m,l) := \G(m,l+\tau) - \G(m,l),
\]
where $\tau \in \tilde{\Rc}(l)$ and } $\G$ is the lattice Green's function for the anti-plane crack geometry, as introduced in \cite[Theorem 2.6]{2018-antiplanecrack} and proven to satisfy decay property
\begin{equation}\label{d1d2G}
|\tilde{D}_1 D_{2\tau} \G(m,l){\cb |} \lesssim (1+|\omega(m)|\,|\omega(l)|\,|\omega(m)-\omega(l)|^{2-\delta})^{-1},
\end{equation}
where $\omega$ is the complex square root mapping defined in polar coordinates as
\[
\omega(l) := r^{1/2}(\cos(\theta/2),\sin(\theta/2)),\,\quad {\cb \theta \in (-\pi,\pi)}
\]
and $\delta > 0$ is arbitrarily small. Here, and throughout this proof,
$\lesssim$ should be read as $\leq C_\delta$ where $C_\delta$ is a constant that may depend on $\delta$.

{\it Proof of Theorem \ref{thm:us}: estimate \eqref{usdecay}: } We first prove the decay estimate for $\bar{u}_s$. We can write
\begin{align}
D_{\tau}\bar{u}_s(l)
&=
\sum_{m \in \La} \tilde{D}\bar{u}_s(m) \cdot \tilde{D}{v}(m)\nonumber\\
&= \sum_{m\in\La}\sum_{\rho \in \tilde{\Rc}(m)} (D_{\rho}\bar{u}_s(m) - \phi'(D_{\rho}\hat{u}_{\bar{k}_s}(m) + D_{\rho}\bar{u}_s(m)))D_{\rho}{v}(m)\label{decay-across-crack}\\
&\qquad + \sum_{m\in\Gamma_{\pm}}\sum_{\rho \in \Rc\setminus\tilde{\Rc}(m)} (- \phi'(D_{\rho}\hat{u}_{\bar{k}_s}(m) + D_{\rho}\bar{u}_s(m)))D_{\rho}{v}(m).\label{decay-across-crack2}
\end{align}
The term \eqref{decay-across-crack} can be estimated by $|l|^{-3/2+\delta}$ due to the argument given in \cite[Theorem 2.4]{2018-antiplanecrack}; that is,
\[
	\bigg|
	\sum_{m\in\La}\sum_{\rho \in \tilde{\Rc}(m)} (D_{\rho}\bar{u}_s(m) - \phi'(D_{\rho}\hat{u}_{\bar{k}_s}(m) + D_{\rho}\bar{u}_s(m)))D_{\rho}{v}(m)
	\bigg|
	\lesssim
	|l|^{-3/2+\delta}.
\]

The additional term \eqref{decay-across-crack2} appears because we define the energy with the homogeneous discrete gradient operator {\cb and concerns the bonds crossing the crack surface. The estimate in \eqref{d1d2G} thus does not apply, but this term can be estimated as follows.} Using \eqref{uhat-across-crack} we see that we only sum over at most $2R_{\phi}^{3}$ lattice sites in \eqref{decay-across-crack2} and thus we can decompose $D_{\rho}v(m)$ into a sum of finite differences along bonds that go around the crack. There will be at most $2R_{\phi}^3$ many of them and each separately decays like $|\Dm{\rho}\Ds{\tau}\G(m,l)| \lesssim |l|^{-3/2+\delta}$, thus ensuring that \eqref{decay-across-crack2} can be bounded by
\[
	\bigg|
	\sum_{m\in\Gamma_{\pm}}\sum_{\rho \in \Rc\setminus\tilde{\Rc}(m)} (- \phi'(D_{\rho}\hat{u}_{\bar{k}_s}(m) + D_{\rho}\bar{u}_s(m)))D_{\rho}{v}(m)
	\bigg|
		\lesssim C |l|^{-3/2+\delta},
\]
where $C < 4R_{\phi}^6$ (though an optimal $C$ is much smaller). This concludes the proof of \eqref{usdecay}.

{\it Proof of Theorem \ref{thm:us}: estimate \eqref{gammasdecay}: } To estimate $\gamma_s$ we employ an analogous argument. We begin by writing
\begin{align}
D_{\tau}\gamma_s(l) &=
 \sum_{m} \tilde{D}\gamma_s(m) \cdot \tilde{D}{v}(m)\nonumber\\
&= \sum_{m\in\La}\sum_{\rho \in \tilde{\Rc}(m)} (D_{\rho}\gamma_s(m) - \phi''(D_{\rho}\hat{u}_{\bar{k}_s}(m) + D_{\rho}\bar{u}_s(m))D_{\rho}\gamma_s(m))D_{\rho}{v}(m)\label{decay-across-crack3}\\
& \qquad + \sum_{m\in\Gamma_{\pm}}\sum_{\rho \in \Rc\setminus\tilde{\Rc}(m)} (- \phi''(D_{\rho}\hat{u}_{\bar{k}_s}(m) + D_{\rho}\bar{u}_s(m))D_{\rho}\gamma_s(m))D_{\rho}{v}(m).\label{decay-across-crack4}
\end{align}
Using precisely the same argument as for \eqref{decay-across-crack2} we can bound
\eqref{decay-across-crack4} by $|l|^{-3/2+\delta}$.

To estimate \eqref{decay-across-crack3} we Taylor-expand $\phi''$ around 0 and{\cb , noting that we assume that $\phi''(0)=1$, we} observe that
\begin{align}
&\left|\sum_{m\in\La}\sum_{\rho \in \tilde{\Rc}(m)} (D_{\rho}\gamma_s(m) - \phi''(D_{\rho}\hat{u}_{\bar{k}_s}(m) + D_{\rho}\bar{u}_s(m))D_{\rho}\gamma_s(m))D_{\rho}{v}(m)\right| \label{decaygamma_aux} \\
&\lesssim \|\gamma_s\|_{\Hcc}\left(\sum_{m\in\La} |R(m)|^2|\tilde{D}v(m)|^2 \right)^{1/2} \nonumber,
\end{align}
where $|R(m)| \lesssim |m|^{-1}$ is the remainder of the expansion. This readily implies that $|D_{\tau}\gamma_s(l)| \lesssim |l|^{-1}$. Thus looking again at \eqref{decaygamma_aux}, instead of applying Cauchy-Schwarz inequality, we directly observe that
\begin{align*}
&\left|\sum_{m\in\La}\sum_{\rho \in \tilde{\Rc}(m)} (D_{\rho}\gamma_s(m) - \phi''(D_{\rho}\hat{u}_{\bar{k}_s}(m) + D_{\rho}\bar{u}_s(m))D_{\rho}\gamma_s(m))D_{\rho}{v}(m)\right| \\
&\lesssim \sum_{m\in\La} |R(m)|{\cb|}\tilde{D}\gamma_s(m)||\tilde{D}v(m)| \lesssim |l|^{-3/2+\delta},
\end{align*}
since $|R(m)|{\cb|}\tilde{D}\gamma_s(m)| \lesssim |m|^{-2}$. As before, $\delta > 0$ is arbitrarily small. This completes the proof of the second bound \eqref{gammasdecay}.
\end{proof}

\begin{remark}
It is interesting to note that while the model includes a full interaction between nearest-neighbour atoms, even across the crack, it is nonetheless the lattice Green's function for the fractured domain that is employed to estimate the atomistic solutions. The homogeneous lattice Green's function fails because the finite differences of $\hat{u}_k(m)$ across the crack grow like $\sim |m|^{1/2}$.
\end{remark}

\subsection{Convergence Proofs}\label{p-conv}
In tandem with the results from bifurcation theory stated in Section \ref{sec:pre}, in order to prove the results from Section \ref{sec:approx} we rely on the following auxiliary result from \cite{EOS2016} that was adapted to domain with cracks in \cite{2018-antiplanecrack}.
\begin{lemma}
There exists a truncation operator $T_R\,:\, \Hcc \to \Hc_R^0$ such that $T_R v = 0$ in $\Lambda \setminus B_R$ and which satisfies
\begin{equation}\label{tranc-ineq}
\|T_R v - v\|_{\Hcc} \lesssim \|v\|_{\Hcc(\La \setminus B_{R/2})} := \left(\sum_{m \in \La \setminus B_{R/2}} |\tilde{D}v(m)|^2\right)^{1/2}\quad \forall v \in \Hcc.
\end{equation}
\end{lemma}
We can now prove the main result of this section.
\begin{proof}[Proof of Theorem \ref{approx-thm-1}]
We consider an extended system $F\,:\,B\times Y \to Z$ where $B = [0,1]$, $Y = \Hc^0_R \times \R$ and $Z = (\Hc^0_R)^* \times \R$ given by
\begin{equation}\label{Fdef}
F(s,y) = (\delta_u \E(u_y,k_y), (u_y - \bar{u}_s,\bar{u}_s')_{\Hcc}),
\end{equation}
where $\bar{u}_s' = \frac{d \bar{u}(s)}{ds}$ and $y = (u_y,k_y)$. We further introduce a mapping $y_R\,:\,B \to Y$ given by $y_R(s) = (T_R\bar{u}_s,\bar{k}_s)$. We shall now show that, with the help of ABCD Lemma, $F$ satisfies the conditions of Theorem~\ref{thm:brezzi1}.

One can easily obtain that
\[
D_y F(s,y_R(s)) = \begin{pmatrix}
  \delta^2_{uu} \E(T_R\bar{u}_s,\bar{k}_s) & \delta^2_{uk} \E(T_R\bar{u}_s,\bar{k}_s) \\
  (\bar{u}_s',\cdot)_{\Hcc} & 0
 \end{pmatrix} =: \begin{pmatrix}
  A_s^R & b_s^R \\
  (\bar{u}_s',\cdot)_{\Hcc} & 0
 \end{pmatrix}.
\]
We further define $A_s:= H_s = \delta^2_{uu} \E(\bar{u}_s,\bar{k}_s)$ (renamed to keep intuitive notation) and $b_s := \delta^2_{uk} \E(\bar{u}_s,\bar{k}_s)$ and observe that
\[
D_y F(s,y_R(s)) = \begin{pmatrix}
  A_s & b_s \\
  (\bar{u}_s',\cdot)_{\Hcc} & 0
 \end{pmatrix} + \begin{pmatrix}
  A_s^R-A_s & b_s^R-b_s \\
  0 & 0
 \end{pmatrix} =: M_s^1 + M^2_s.
\]
Here, we treat $A_s$ as a restriction to $\Hc_0^R$ and $b_s$ as an element of $(\Hc_0^R)^*$. Since $T_R \bar{u}_s \to \bar{u}_s$ as $R \to \infty$ strongly in $\Hcc$ (a consequence of \eqref{tranc-ineq}, the decay estimate from Theorem \ref{thm:us} and $\E \in C^{\alpha}$),
\begin{equation} \label{eq:prf:convergence_As_Bs}
	\| A_s^R - A_s \|_{\mathcal{L}(\Hcc,(\Hcc)^*)} + \| b_s^R - b_s \|_{(\Hcc)^*} \to 0
\end{equation}
as $R \to \infty$.
Our strategy will therefore be to apply the ABCD Lemma to $M_s^1$, interpreted as an operator from $\Hcc \times \R$ to $(\Hcc)^*\times \R$, and show that $M_s^2$ is a small perturbation. 

To carry out this strategy we begin by differentiating  $\delta_u\E(\bar{u}_s,\bar{k}_s) = 0$ with respect to $s$ to obtain
\begin{equation}\label{Asbs}
A_s \bar{u}_s' + \bar{k}_s' b_s = 0
\end{equation}
along the bifurcation path $\B$. At a fold point, when $s=b_i$, due to \eqref{nondeg2}, we have $\bar{k}_s' = 0$, thus revealing that $\bar{u}_{b_i}' = \alpha \gamma_{b_i}$ for some non-zero $\alpha \in \R$. For $s\neq b_i$, the operator $A_s$ is invertible and thus
\begin{equation}\label{usdash}
\bar{u}_s' = -\bar{k}_s' (A_s)^{-1}b_s.
\end{equation}
We can now show that $M_s^1$ satisfies the conditions of ABCD Lemma.

{\it Case 1, $s \in \mathcal{I}_{\rm pos}$: }
Suppose that $s \in \mathcal{I}_{\rm pos}$ from \eqref{S_st_unst}. In this case $A_s$ is an isomorphism due to \eqref{stab2-new} and \eqref{Us-stable}. Thus, to apply ABCD Lemma to $M_s^1$, we have to check that $(\bar{u}'_s,(A_s)^{-1}b_s)_{\Hcc} \neq 0$, which is true since
\[
(\bar{u}'_s,(A_s)^{-1}b_s)_{\Hcc} = -\bar{k}_s' (\bar{u}_s',\bar{u}_s')_{\Hcc} \neq 0,
\]
since by definition at a regular point we have $\bar{k}_s' \neq 0$ and $\bar{u}_s' \neq 0$.

{\it Case 2, $s \in \mathcal{I}_{\rm pt}$: } Now suppose $s \in \mathcal{I}_{\rm pt}$ but $s \neq b_i\;\forall i \in \{1,\dots,M\}$. It can be shown that $A_s$ remains an isomorphism as follows. Proposition~\ref{prop-eigen} tells us that we have an eigen-pair $(\mu_s,\gamma_s)$ satisfying \eqref{eigen-pair}. Any $v \in \Hcc$ can be decomposed into $v = \alpha \gamma_s + w$, where $w \in U_s$ with $U_s$ given by \eqref{Us-unstable} and $\alpha \in \R$. Thus,
\[
\|A_s v\| =  \sup_{\substack{\tilde{v} \in \Hcc \\ \|\tilde{v}\| = 1}} |\<{A_s v}{\tilde{v}}| \geq \frac{1}{2}\left(|\<{A_s v}{\gamma_s}| + \big|\<{A_s v}{\textstyle \frac{w}{\|w\|}}\big|\right),
\]
and we can further estimate
\begin{align*}
|\<{A_sv}{\gamma_s}| &= \big|\alpha\<{A_s \gamma_s}{\gamma_s} + \<{A_s w}{\gamma_s}\big| = |\alpha \mu_s|,
\\
|\<{A_s v}{\textstyle\frac{w}{\|w\|}}| &= |\alpha\<{A_s \gamma_s}{\textstyle\frac{w}{\|w\|}} + \<{A_s w}{\textstyle\frac{w}{\|w\|}}| \geq \frac{c}{\|w\|} \|w\|^2 = c\|w\|,
\end{align*}
where $c$ is the stability constant from Assumption \ref{ass2-new}. This, together with the fact that $\|v\| \leq |\alpha| + \|w\|$ readily implies that we can set $\tilde{c} := \frac{1}{2} \min\{|\mu_s|, c\}$ and conclude that for all $v \in \Hcc$
\[
\|A_s v\| \geq \tilde{c}\|v\|.
\]
Thus, as in the case $s \in \mathcal{I}_{\rm pos}$, \eqref{usdash} ensures that we can apply the ABCD lemma and deduce again that $A_s$ is an isomorphism.

{\it Case 3, $s = b_i$: } Finally, suppose $s = b_i$ for some $i \in \{1,\dots,M\}$. Due to Assumption \ref{ass2-new} we know that the kernel of $A_s$ is one-dimensional at a fold point and thanks to Proposition \ref{prop-eigen} we know that it is spanned by $\gamma_s$, which means that \eqref{Asbs} implies that $\bar{u}_s' = \gamma_s$. By Definition \ref{deffp} we know that $\<{b_s}{\gamma_s} \neq 0$, which implies that the ABCD Lemma is again applicable.

{\it Uniform Stability of $M_s^1$: } We have shown so far that, for all $s \in [0, 1]$, $M_s^1$ is an isomorphism from $\Hcc \times \R$ to $(\Hcc)^* \times \R$. In particular, this implies that for any $x = (u_{x},k_{x}) \in \Hcc \times \R$ we have
\[
\|M_s^1x\| \geq \tilde{c}_s \|x\|,
\]
where $\tilde{c}_s > 0$.

Since $s \mapsto M_s^1$ is continuous in operator-norm due to smoothness of $\E$ established in Theorem \ref{Eu-well-def} and smoothness of $s\mapsto (\bar{u}_s,\bar{k}_s)$ established in Proposition \ref{prop-b}, it follows that the infimum $\inf \tilde{c}_s$ is attained on $[0, 1]$ and must therefore be positive. In summary, we have established the existence of $\tilde{c} > 0$ such that
\[
 \|M_s^1 x\| \geq \tilde{c} \|x\| \qquad
 	\forall s \in [0, 1], \quad x \in \Hc_0^R \times \R.
\]

{\it Uniform Stability: } Next, using the definition of $M_s^2$ we can bound
\begin{align*}
\|D_yF(s,y_R(s)) x\| &\geq \|M_s^1 x\| - \|M_s^2 x\| \geq \tilde{c}\| x\| - \|A_s^R - A_s\|\|u_{x}\| - \|b_s^R-b_s\||k_x| \\
	&\geq \frac{\tilde{c}}{2}\|x\|,
\end{align*}
for $R$ large enough, thus ensuring that $D_yF(s,y_R(s))$ is an isomorphism from $\Hc_0^R \times \R$ to $(\Hc_0^R)^* \times \R$, thus satisfying condition (i) from Theorem \ref{thm:brezzi1}, with uniform bound
\[
	\|D_yF(s,y_R(s)) x\| \geq \frac{\tilde{c}}{2}\|x\|
		\qquad \forall s \in [0, 1], \xi \in \Hc_0^R \times \R,
\]
that is $c_0$ from \eqref{c0} is given by $c_0= \frac{2}{\tilde{c}}$.

{\it Conclusion: } So far we have confirmed Condition (i) of
Theorem \ref{thm:brezzi1}. To conclude the proof, we now need to also check
conditions (ii, iii).

It can be readily checked that $\|D_s F(s,y_R(s))\| = | -1 + (T_R \bar{u}_s - \bar{u}_s,\bar{u}_s'')_{\Hcc}| \leq 2$ for $R$ large enough. Thus the condition (ii) in Theorem \ref{thm:brezzi1} is satisfied with $c_1$ in \eqref{c1} given by $c_1 = 2$.

The condition (iii) from Theorem \ref{thm:brezzi1} is satisfied with
\[
L_1(\xi) := \sup_{(s_*,y_*)\in S(s,y_R(s),\xi)} \|D^2F(s_*,y_*)\|,
\]
where $S(s,y_R(s),\xi) = \{ (s_0,y_0) \in \R \times (\Hc_0^R \times \R) \, \colon \, |s - s_0| + \|T_R\bar{u}_s - u_y\| + |\bar{k}_s - k_y| \leq \xi \}$.

Finally, we observe that
\[
\sup_{s \in [0,1]}\|F(s,y_R(s))\| = \sup_{s \in [0,1]} \left(\|\delta_u \E(T_R\bar{u}_s,\bar{k}_s)\| + |(T_R\bar{u}_S - \bar{u}_s,\gamma_s)|\right) \to 0,
\]
as $R \to \infty$, which implies that no matter how large constants $c_0,c_1$ and $c_2$ were and how badly behaved $L_1$ was, we would still fall within the regime where the result of Theorem \ref{thm:brezzi1} was applicable for $R$ large enough.

We can thus conclude that there exists $\B_R\,:\,[0,1]\to \Hc_0^R\times \R$ given by $\B_R(s):=(\bar{u}^R_s,\bar{k}_s^R)$, such that $F(s,(\bar{u}^R_s,\bar{k}_s^R)) = 0$, which in particular implies $\delta_u \E(\bar{u}^R_s,\bar{k}_s^R) = 0$. Furthermore, using \eqref{brezzi-ineq} we can conclude that
\[
\|\bar{u}_s^R - T_R\bar{u}_s\|_{\Hcc} + |\bar{k}_s^R - \bar{k}_s| \leq K_0\|F(s,y_R(s))\| \lesssim \|T_R\bar{u}_s - \bar{u}_s\|\lesssim R^{-1/2+\beta},
\]
for arbitrarily small $\beta >0$. Crucially, $K_0$ depends only on $c_0$ and $c_1$, which are independent of $s$ and the last inequality follows from Lemma \ref{tranc-ineq} and the regularity estimate from Theorem \ref{thm:us}. This concludes the result, since trivially $\|\bar{u}_s^R - \bar{u}_s\|_{\Hcc} \leq \|\bar{u}_s^R - T_R\bar{u}_s\|_{\Hcc} + \|T_R\bar{u}_s^R - \bar{u}_s\|_{\Hcc}$.

Finally, we note that \eqref{BR_norm_conv_est} follows as an immediate collorary, arguing exactly as in the proof of \cite[Theorem 2.4]{EOS2016pp}.
\end{proof}

To prove our final result, the superconvergence of critical values of SIF, we first need to quote two intermediate technical steps. The first lemma, which highlights the origin of this superconvergence, is taken from \cite[Theorem 4.1]{cliffe_spence_tavener_2000}, restated in our notation for the sake of convenience.

\begin{lemma}\label{fplemma}
Let $(\bar{u}_{b_i},\bar{k}_{b_i}) \in \B$ be a simple quadratic fold point. Under Assumptions \ref{ass1},\ref{ass2-new} \& \ref{ass3}, for $R$ large enough, the approximate bifurcation diagram $\B_R$ has a quadratic fold point at $s=b_i^R$, where $|b_i^R-b_i| \to 0$ as $R \to \infty$. Furthermore,
\begin{align*}
|\bar{k}_R(b^R_i) - \bar{k}(b_i)| &\leq
\left\| \bar{u}_R'(b_i) - \bar{u}'(b_i) \right\|_{\Hcc}^2
+ \left|\bar{k}_R'(b_i) -\bar{k}'(b_i)\right|^2 \\
&\qquad + \|\bar{u}_R(b_i) - \bar{u}(b_i)\|_{\Hcc}^2
+ |\bar{k}_R(b_i) -\bar{k}(b_i)|^2 \\
&\qquad + \inf_{v \in \Hc_0^R}
	\|v - \gamma_{b_i}\|_{\Hcc}^2.
\end{align*}
\end{lemma}
To exploit the inequality from Lemma \ref{fplemma}, we adapt \cite[Theorem 2]{Brezzi1}, which is a follow-up result to Theorem \ref{thm:brezzi1} for derivatives.
\begin{lemma}\label{brezzider}
Assume the hypotheses of Theorem \ref{thm:brezzi1} and in addition that
\[
\sup_{x_0 \in U}\|D F(x_0,y(x_0){\cb)}\| \leq c_1.
\]
Then there exists a continuous function $K\,:\,\R_+ \to \R_+,$ which depends only on $c_0,c_1,L_1$ such that for all $x_0 \in U$ and all $x \in B(x_0,a)$  it holds that
\begin{align*}
\|Dg(x) - Dy(x_0)\| \leq K( \|Dy(x_0)\|) \Big( & \|x-x_0\| + \|F(x_0,y(x_0){\cb)}\| \\
									   & + \|DF(x_0,y(x_0){\cb)},Dy(x_0)\| \Big).
\end{align*}
\end{lemma}
\begin{proof}[Proof of Theorem \ref{approx-thm-2}]
		Lemma \ref{fplemma} implies that, for domain radius $R$ large enough,
		 we have exactly $M$ approximate fold points $b_j^R \to b_j$ as $R \to \infty$.
		 By assumption, $b_j \in (0, 1)$ and hence also $b_j^R \in (0,1)$. Arguing analogously as in the proof of Theorem \ref{approx-thm-1}, it is not difficult to show that $F$ defined in \eqref{Fdef} satisfies the conditions of Lemma \ref{brezzider}, thus
\[
\left\| \bar{u}_R'(b_i) - \bar{u}'(b_i)\right\| + \left| \bar{k}_R'(b_i) -\bar{k}'(b_i) \right| \lesssim R^{-1/2+\beta},
\]
for arbitrary small $\beta > 0$. Furthermore,
\[
\inf_{v \in \Hc_0^R}
	\|v - \gamma_{b_i}\|_{\Hcc} \leq \|T_R\gamma_{b_i} - \gamma_{b_i}\| \lesssim R^{-1/2+\beta},
\]
with the first inequality following from the obvious fact that $T_R\gamma_{b_i} \in {\cb \Hc^0_R}$ and the second from the regularity result for $\gamma_{b_i}$ in Theorem \ref{thm:us}.

Applying Lemma~\ref{fplemma} we therefore obtain the desired result that
\[
|\bar{k}_R(b^R_i) - \bar{k}(b_i)| \lesssim R^{-1+\beta}.
\qedhere
\]
\end{proof}

\bibliographystyle{myalpha} 
\bibliography{papers}
\end{document}